\documentclass[11pt,a4paper]{article}
\usepackage[top=3cm, bottom=3cm, left=3cm, right=3cm]{geometry}
\usepackage[latin1]{inputenc}
\usepackage{amsmath}
\usepackage{amsthm}
\usepackage{amsfonts}
\usepackage{amssymb}
\usepackage{graphicx}
\usepackage{tabu}
\usepackage{multicol, array}
\usepackage{enumerate}
\usepackage{enumitem}
\usepackage{hyperref}

\DeclareMathOperator{\li}{li}
\DeclareMathOperator{\ord}{ord}
\DeclareMathOperator{\ind}{ind}

\DeclareMathOperator{\lcm}{lcm}

\numberwithin{equation}{section}

\newtheorem{thm}{Theorem}[section]
\newtheorem{lem}{Lemma}[section]
\newtheorem{conj}{Conjecture}[section]

\newtheorem{cor}{Corollary}[section]
\newtheorem{dfn}{Definition}[section]

\newcommand{\N}{\mathbb{N}}
\newcommand{\Z}{\mathbb{Z}}
\newcommand{\Q}{\mathbb{Q}}

\newcommand{\F}{\mathbb{F}}
\usepackage{fancyhdr}
\title{Results for Wieferich Primes }
\date{}
\author{N. A. Carella}

\begin{document}
\maketitle

\maketitle
\vskip .25 in 
\textbf{Abstract:}  
Let $v\geq 2$ be a fixed integer, and let $x \geq 1$ and $z \geq x$ be  large numbers. The exact asymptotic formula for the number of Wieferich primes $p$, defined by $ v^{p-1} \equiv 1 \bmod p^2$, in the short interval $[x,x+z]$ is proposed in this note. The search conducted on the last 100 years have produced two primes $p<x=10^{15}$ such that $2^{p-1} \equiv 1 \bmod p^2$. The probabilistic and theoretical information within predicts the existence of another base $v=2$ prime on the interval $[10^{15},10^{40}]$.  Furthermore, a result for the upper bound on the number of Wieferich primes is used to demonstrate that the subset of nonWieferich primes has density 1.

\vskip .25 in 
 \textbf{AMS Mathematical Subjects Classification:} Primary 11A41; Secondary 11B25.\\
\textbf{Keywords:} Distribution of Prime, Wieferiech prime, Finite Rings.\\

\tableofcontents

%11111111111111111111111111111111111111111111111
\newpage
\section{Introduction} \label{s1}
Let $p\geq 3$ denotes a prime, and let $v \geq 2$ be a fixed integer base. The set of Wieferich primes is defined by the congruence $v^{p-1} \equiv 1 \bmod p^2$,  see \cite[p.\ 333]{RP98}. These numbers are of interest in Diophantine equations, see \cite{WA09}, \cite[Theorem 1]{MP03}, algebraic number theory, the theory of primitive roots, see \cite{PA09}, additive number theory, see \cite{GS99}, and many other topics in mathematics. \\ 

In terms of the order of the element $v \in \left ( \mathbb{Z} / p^2 \mathbb{Z} \right )^{\times}$ in the finite ring, this set has the equivalent description
\begin{equation} \label{6000}
	\mathcal{W}_v=\left\{ p:\ord_{p^2}(v) \mid p- 1 \right \}.
\end{equation}
 For a large number $x \geq 1$, the corresponding counting function for the number of Wieferich primes up to $x$ is defined by
\begin{equation} \label{6002}
	W_{v}(x)=\#\left\{ p\leq x:\ord_{p^2}(v) \mid p-1 \right \}.
\end{equation}
The heuristic argument in \cite[p.\ 413]{RP98}, \cite{CP97}, et alii, claims that
\begin{equation} \label{6006}
	W_{v}(x) \approx  \sum_{p \leq x} \frac{1}{p} \ll \log \log x.
\end{equation}
The basic idea in this heuristic was considerably improved and generalized in \cite[Section 2]{KN15}. The conditional analysis is based on some of the statistical properties of the Fermat quotient. Specifically, this is a map
\begin{equation}
\begin{array} {cll}
\left ( \mathbb{Z} / p^2 \mathbb{Z} \right )^{\times} & \longrightarrow & \mathbb{F}_p, \\
v& \longrightarrow &q_v(p), \\
\end{array} 
\end{equation}
defined by
\begin{equation} \label{6000}q_v(p)\equiv \frac{v^{p-1}-1}{p} \bmod p^2 \nonumber .
\end{equation} 
Each integer $v \in \N$ is mapped into an infinite sequence 
\begin{equation}
\{x_p(v): p\geq 2 \}=(q_v(2),q_v(3),q_v(5), \ldots ) \in \mathcal{P}(v).
\end{equation}
 The product space $\mathcal{P}(v)=\prod_{p \nmid v} \F_p$ is the set of sequences $\{x_p(v)=q_v(p): p \nmid v\}$.  Evidently, the subset of sequences $\{x_p(v)=q_v(p)=0: p \nmid v\}$ is equivalent to $\mathcal{W}_v$, and the counting function is
\begin{equation} \label{6002}
	W_{v}(x)=\#\{p \leq x: p \nmid v, \; x_p(v) \equiv 0 \bmod p \} .
\end{equation} 

%sssssssssssssssssssssssssssssssssssssssssssssssss
\subsection{Summary of Heuristics}
A synthesis of some of the previous works, such as Artin heuristic for primitive roots, and the heuristics arguments in \cite{RP98}, \cite{CP97}, \cite{KN15}, et alii, are spliced together here. 
\begin{conj} \label{conj1.1} For any integer $v \geq 2$, the subset of elements $\{x_p(v)\} \in \mathcal{P}(v)$ for which $x_p(v)=0$ has the asymptotic formula
\begin{equation}
 \#\{p \leq x: p \nmid v, \; x_p(v) \equiv 0 \bmod p \} = c_v \log \log x + o((\log \log x)^{(1+\varepsilon)/2} ),
\end{equation}
where $\epsilon >0$ is a small number, and the correction factor is defined by 
\begin{equation} \label{666}
c_v= \sum_{n \geq 1}\sum_{d \mid n} \frac{\mu(n)\gcd(dn,k)}{ dn \varphi(dn) } \nonumber ,
\end{equation}
where $v=ab^k$ with $a\geq 1$ squarefree; as $x \to \infty$. 
\end{conj}
The correction factor $c_v \geq 0$ accounts for the dependencies among the primes. According to the analysis in \cite{KN15}, for any random integer $v\geq 2$, the correction factor 
$c_v=1$ with probability one. Thus, $c_v\ne 1$ on a subset of integers $v\geq 2$ of zero density. For example, at the odd prime powers $v \equiv 1 \bmod 4$, see Theorem \ref{thm77.1} for more details.\\

%ssssssssssssssssssssssssssssssssssssssssssssssssssss
\subsection{Results In Short Intervals}
The purpose of this note is to continue the investigation of the asymptotic counting functions for Wieferich primes in (\ref{6002}). A deterministic analysis demonstrates that the number of Wieferich primes is very sparse, and it is infinite. \\

\begin{thm} \label{thm1.1} Let $v\geq 2$ be a base, and let  $x \geq 1$ and $z \geq x$ be large numbers. Then, the number of Wieferich primes in the short interval $[x, x+z]$ has the asymptotic formula
\begin{equation} \label{6022}
	W_v(x+z)-	W_v(x)=c_v \left ( \log \log (x+z)-\log \log( x) \right )+E_v( x,z),
\end{equation}
where $c_v \geq 0$ is the correction factor, and $E_v(x,z)$ is an error term.
\end{thm} 
This is equivalent to the asymptotic form
\begin{equation} \label{6025}
W_v(x) \sim \log \log( x),                
\end{equation}
and demonstrates that the subset of primes $\mathcal{W}_v$ is infinite.
\begin{thm} \label{thm1.2} Let $v\geq 2$ be a base, and let $x \geq 1$ be a large number. Then, the number of Wieferich primes on the interval $[1, x]$ has the upper bound
\begin{equation} \label{6022}
	W_{v}(x) \leq 4v \log \log x .
\end{equation}
\end{thm} 
Sequences of integers divisible by high prime powers fall on the realm of the $abc$ conjecture. Thus, it is not surprising to discover that the sequence of integers $s_p=2^{p-1}-1, \text{ prime } p \geq2,$ has finitely many terms divisible by $p^3$.   
\begin{thm} \label{thm1.3} Let $v \geq 2$ be a fixed integer base. Then, the subset of primes 
\begin{equation} \label{6010}
	\mathcal{A}_v=\left\{ p:\ord_{p^3}(v) \mid p- 1 \right \}
\end{equation}
is finite. In particular, the congruence $v^{p-1}-1 \equiv 0 \bmod p^3$ has at most finitely many primes $p\geq 2$ solutions.
\end{thm} 

%ssssssssssssssssssssssssssssssss
\subsection{Average Order}
The average orders of random subsets $\mathcal{W}_v $ has a simpler to determined asymptotic formula, and the average correction factor $c_v=1$.  

\begin{thm} \label{thm1.4} Let $x \geq 1$ and $z \geq x$ be large numbers, and let $v \geq 2$ random integer. Then, the average order of the number of Wieferiech primes has the asymptotic formula
\begin{equation} \label{6022}
W_v(x+z)-	W_v(x)= \log \log (x+z)-\log \log( x) +E_v( x,z)                 
\end{equation}
where $E_v(x,z)$ is an error term.
\end{thm} 
The asymptotic form
\begin{equation} \label{6025}
W_v(x) \sim \log \log( x)                 
\end{equation}
is evident in this result.

%sssssssssssssssssssssssssssssssssssssssssssss
\subsection{Guide}
Sections \ref{s22} to \ref{s55} provide the basic foundation, and additional related results, but not required. The proofs of Theorems \ref{thm1.1}, \ref{thm1.2}, and \ref{thm1.3} appear in Section \ref{s66}, and the proof of Theorem \ref{thm1.4} appears in Section \ref{s9}. Section \ref{s77} is concerned with the derivation of the correction factor $c_v\geq 0$. Estimates for the next Wieferiech primes to bases $v=2$ and $v=5$ are given in Section \ref{s11}. A proof for the density of nonWieferiech primes appears in Theorem \ref{thm13.1}.\\

These ideas and Conjecture \ref{conj1.1} take the more general forms
\begin{eqnarray} \label{6002}
W_{n,v}(x)&=&\# \left\{ p\leq x:v^{p^{n-1}(p-1)}-1 \equiv 0 \bmod p^{n+1} \right \} \\
&=& c_{n,v} \log \log x + o((\log \log x)^{(1+\varepsilon)/2} ) \nonumber
\end{eqnarray}
for $n \geq 1$, and the subset $\mathcal{A}_{n,v}=\left\{ p:\ord_{p^{n+2}}(v) \mid p- 1 \right \}$ is finite.\\

In synopsis, the sequence of integers investigated here have the followings divisibility properties.
\begin{enumerate} 
\item The sequence of integers $2^{p-1}-1$ is divisible by every prime $p \geq 2$, for example, $2^{p-1}-1 \equiv 0 \bmod p$. The corresponding subset of primes has density 1, this follows from Fermat little theorem.
\item  The sequence of integers $2^{p-1}-1$ is divisible by infinitely many prime powers $p^2 \geq 4$, for example, $2^{p-1}-1 \equiv 0 \bmod p^2$.  The corresponding subset of primes has zero density, this follows from Theorem \ref{thm1.1}.
\item The sequence of integers $2^{p-1}-1$ is divisible by finitely many prime powers $p^3 \geq 8$, for example, $2^{p-1}-1 \equiv 0 \bmod p^3$. The corresponding subset of primes has zero density, this follows from Theorem \ref{thm1.3}.\\

Other examples of interesting sequences of integers that might have similar divisibility properties are the followings. \\ 

\item For fixed pair $a>b>0$, as $n \to \infty$, the sequence of integers $a^{n}-b^n$ is divisible by infinitely many primes $p \geq 2$, for example, 

$a^{n}-b^n\equiv 0 \bmod p$. The corresponding subset of primes has nonzero density, this follows from Zsigmondy theorem, see also \cite{LJ94}, \cite{BV01}, et cetera. 
\item For fixed pair $a>b>0$, as $n \to \infty$, the sequence of integers $a^{n}-b^n$ is probably divisible by infinitely many prime powers $p^2$, for example, $a^{n}-b^n\equiv 0 \bmod p^2$. This is expected to be a subset of primes of zero density, but there is no proof in the literature.
\item For a fixed pair $a>b>0$, as $n \to \infty$, the sequence of integers $a^{n}-b^n$ is probably divisible by finitely many prime powers $p^3$, for example,  $a^{n}-b^n\equiv 0 \bmod p^3$. This is expected to be a a subset of primes of zero density, but there is no proof in the literature.
\end{enumerate}
More general and deeper information on the prime divisors of sequences and prime power divisors of sequences appears in \cite[Chapter 6]{EW03}, and the references within.

%222222222222222222222222222222222222222222222222222222222222222222222222
\newpage
\section{Basic Analytic Results } \label{s22}
A few analytic concepts and results are discussed in this Section. \\
The logarithm integral is defined by
\begin{equation}  \label{22.505}
\li(x)=\int_2^x \frac{1}{\log t} dt=\frac{x}{\log x} + a_2\frac{x}{\log^2 x} + a_3\frac{x}{\log^3 x} + a_4\frac{x}{\log^4 x} +O \left (\frac{1}{\log^5 x}\right ).
\end{equation} 
A refined version over the complex numbers appears in \cite[p.\ 20]{MC05}. The corresponding prime couting function has the Legendre form
\begin{equation}  \label{22.507}
\pi(x)=\frac{x}{\log x}+ a_2\frac{x}{\log^2 x} + a_3\frac{x}{\log^3 x} + a_4\frac{x}{\log^4 x} +O \left (\frac{1}{\log^5 x}\right ),
\end{equation}

%ssssssssssssssssssssssssssssssssssssssssssssssssssssssssssssssssssssssss
\subsection{Sums And Products Over The Primes}
\begin{lem} \label{lem22.1} {\normalfont(Mertens)} Let $x \geq 1$ be a large number, then the followings hold.
\begin{enumerate}[font=\normalfont, label=(\roman*)]
\item The prime harmonic sum satisfies 
\begin{equation} \label{4095}
 \sum_{ p\leq x}  \frac{1}{p} =\log \log x+b_0 +\frac{b_1}{\log x}+\frac{b_2}{\log^2 x}+\frac{b_3}{\log^3 x}+O \left (\frac{1}{\log^5 x}\right )
\end{equation}    
where $b_0>0,b_1, b_2,$ and $b_3$ are constants, as \(x \rightarrow \infty\). 
\item The prime harmonic product satisfies
\begin{equation} \label{22.096}
 \prod_{ p\leq x}  \left (1-\frac{1}{p} \right ) =\frac{1}{e^{\gamma}\log x } \left ( 1+O \left (\frac{1}{\log x}\right ) \right ),
\end{equation} 
where $\gamma=0.5772 \ldots $ is a constant.
\end{enumerate}
\end{lem}
where $a_2>0,a_3,$ and $a_4$ are constants,
\begin{proof} (i) Use the Legendre form of the  prime number theorem in (\ref{22.507}).
\end{proof}

More general versions for primes over arithmetic progressions appear in \cite{LZ10}, and a general version of the product appears in the literature. \\

Explicit estimates can be handled with the formulas
\begin{equation} \label{4097}
\left |  \sum_{ p\leq x}  \frac{1}{p} -\log \log x-b_0 \right |< \frac{1}{10\log^2 x}+\frac{4}{15\log^3 x}
\end{equation}    
for $x\geq 10400$ and $b_0$ is a constant; and 
\begin{equation} \label{22.096}
 \frac{1}{e^{\gamma}\log x } \left ( 1-\frac{.2}{\log^2 x} \right )<\prod_{ p\leq x}  \left (1-\frac{1}{p} \right ) <\frac{1}{e^{\gamma}\log x } \left ( 1+\frac{.2}{\log^2 x} \right ),
\end{equation} 
for $x \geq 2$, see \cite{RS62}, \cite{DP99}.

%sssssssssssssssssssssssssssssssssssssssssssssssssssssssssssssssssss
\subsection{Totients Functions}
Let $n=p_1^{e_1} p_2^{e_2} \cdots p_t^{e_t}$ be an arbitrary integer. The Euler totient function over the finite ring $\mathbb{Z}/n\mathbb{Z}$ is  defined by $\varphi(n)= \prod_{p \mid n} \left ( 1 -1/p\right ) $. 
While the more general Carmichael totient function over the finite ring $\mathbb{Z}/n\mathbb{Z}$ is defined by 
\begin{equation}
\lambda(n) 
=\left \{
\begin{array}{ll}
\varphi(2^e), &  n=2^e, e=0,1, \text{ or } 2,\\
\varphi(2^e)/2, & n=2^e, e \geq 3,\\
\varphi(p^e), & n=p^e \text{ or }, 2p^e \text{ and } e \geq 1,
\end{array} \right .
\end{equation}
where $p \geq3$ is prime, and
\begin{equation}
\lambda(n)= \lcm \left ( \lambda(p_1^{e_1})\lambda(p_2^{e_2}), \ldots ,\lambda(p_t^{e_t}) \right ).
\end{equation}

The two functions coincide, that is, \(\varphi(n)=\lambda (n)\) if \(n=2,4,p^m,\text{ or } 2p^m,m\geq 1\). And \(\varphi \left(2^m\right)=2\lambda
\left(2^m\right)\). In a few other cases, there are some simple relationships between \(\varphi (n) \text{ and } \lambda (n)\).

%ssssssssssssssssssssssssssssssssssssssssssssssssssssssss
\subsection{Sums Of Totients Functions Over The Integers}
The Carmichael totient function has a more complex structure than the Euler totient function, however, many  inherited properties such as \\

\begin{tabu} to 1.0\textwidth { X[l]  X[l] }
$\lambda(n) \mid \varphi(n),$ & $\varphi(n)=\xi(n)\lambda(n),$   \\
\end{tabu} \\

can be used to derive information about the Carmichael totient function.

\begin{thm} \label{thm22.1} For all $x\geq 16$, then the followings hold.

\begin{enumerate} [font=\normalfont, label=(\roman*)]
\item The average order has the asymptotic formula
$$
\sum_{n\leq x} \varphi(n)=\frac{3}{\pi^2}x^{2}+O(x \log x).
$$
\item The normalized function  has the asymptotic formula
$$
\sum_{n\leq x}\frac{\varphi(n)}{n}=\frac{6}{\pi^2}x+O(\log x) .
$$
\end{enumerate}
\end{thm}

\begin{thm} \label{thm22.2} {\normalfont(\cite[Theorem 3]{ES91})} For all $x\geq 16$, then the followings hold.
\begin{enumerate} [font=\normalfont, label=(\roman*)]
\item The average order has the asymptotic formula
$$
\frac{1}{x} \sum_{n \leq x} \lambda(n)=\frac{x}{\log x} e^{\frac{B \log \log x}{\log \log \log x} (1 +o(1))},
$$
where the constants are $\gamma=0.5772 \ldots$ and
$$
B=e^{-\gamma} \prod_{p \geq 2} \left( 1-\frac{1}{(p-1)^2(p+1)} \right )=0.34537 \ldots.
$$
\item The normalized function  has the asymptotic formula
$$
 \sum_{n \leq x} \frac{\lambda(n)}{n}=\frac{x}{\log x} e^{\frac{B \log \log x}{\log \log \log x} (1 +o(1))}.
$$
\end{enumerate}
\end{thm}

\begin{proof} (ii) Set $U(x)=\sum_{n \leq x} \lambda(n)$. Summation by part yields
\begin{eqnarray}
\sum_{n \leq x} \frac{\lambda(n)}{n}&=& \int_1^x \frac{1}{t} d U(t) \nonumber \\
&=& \frac{U(x)}{x}-\frac{U(1)}{1}+\int_1^x \frac{ U(t)}{t^2} dt .
\end{eqnarray}
Invoke case (i) to complete the claim.
\end{proof}

\begin{lem} \label{lem22.5}  Let $x\geq 16$, and $\lambda(n)$ be the Carmichael totient function. Then
\begin{equation}
 \sum_{n \leq x} \frac{\varphi(\lambda(n))}{\lambda(n)} \geq  k_1\frac{x}{\log \log x} (1 +o(1)),
\end{equation}
where $k_1>0$ is a constant.
\end{lem}

\begin{proof} Begin with the expression
\begin{eqnarray}
\frac{\varphi(\lambda(n))}{\lambda(n)} &=&\prod_{p \mid \lambda(n)} \left (1- \frac{1}{p} \right ) \nonumber \\
&\geq &\frac{1}{ e^{\gamma}} \frac{1}{\log \log \lambda(n)} \left (1+O\left(  \frac{1}{\log \log \lambda(n)}\right ) \right ) 
\end{eqnarray}
holds for all integers $n \geq 1$, with equality on a subset of integers of zero density. In light of the lower bound and upper bound
\begin{equation}
\frac{n}{(\log n )^{a\log \log \log n}} \leq \lambda(n) \leq \frac{n}{(\log n )^{b\log \log \log n}}
\end{equation} 
for some constants $a>0$ and $b>0$ and large $n \geq 1$, see \cite[Theorem 1]{ES91}, 
proceed to determine a lower bound for the sum: 
\begin{eqnarray}
\sum_{n \leq x} \frac{\varphi(\lambda(n))}{\lambda(n)} &\geq& k_1\int_2^x \frac{1}{\log \log t} d t \nonumber \\
&=& k_1\frac{x}{\log \log x}+k_2\int_2^x \frac{1}{t (\log t) (\log \log t)^2} d t,
\end{eqnarray}
where $k_1>0$ and $k_2$ are constants. This is sufficient to complete the claim.
\end{proof}
More advanced techniques for the composition of arithmetic functions are studied in \cite{MP04}, \cite{BS05}, et cetera.\\

The ratio $\varphi(\varphi(n))/\varphi(n)=\varphi(\lambda(n))/\lambda(n)$ for all integers $n \geq 1$. But, the composition only satisfy $\varphi(\varphi(n))\geq \varphi(\lambda(n))$, and agree on a subset of integers of zero density. 

\begin{lem} \label{lem22.2} {\normalfont(\cite[Theorem 2]{MS04})} Let $x\geq 1$ be a large number. Then
\begin{equation}
\#\{n :\varphi(\varphi(n))\geq \varphi(\lambda(n))\} \ll \frac{x}{(\log \log x)^c}
\end{equation}
for any constant $c>0$.
\end{lem}

%sssssssssssssssssssssssssssssssssssssssssssssssssssssssssssssssssssssssssssssssssssssss
\subsection{Sums Of Totients Functions Over The Primes}
The ratio $\varphi(p-1)/(p-1)$ for all primes $p \geq 2$ has an established literature. But, the ratio $\lambda(p-1)/(p-1)$ does not has any meaningful literature, its average order is estimated here in Lemma \ref{lem22.4}, and a more precise version is stated in the exercises.

\begin{lem} \label{lem22.3} {\normalfont (\cite[Lemma 1]{SP69}) } Let \(x\geq 1\) be a large number, and let \(\varphi (n)\) be the Euler totient 

function. Then
\begin{equation}
\sum_{p\leq x }\frac{\varphi(p-1)}{p-1} =a_0\li(x)+	O\left(\frac{x}{\log ^Bx}\right) ,
\end{equation}
where the constant 
\begin{equation} \label{3500}
a_0=\prod_{p \geq 2 } \left(1-\frac{1}{p(p-1)}\right)=.37399581 \ldots ,
	\end{equation}
and \(\li(x)\) is the logarithm integral, and \(B> 1\) is an arbitrary constant, as \(x \rightarrow \infty\). 
\end{lem}
More general versions of Lemma \ref{lem22.3} are proved in \cite{VR73}, and \cite{HG83}.\\

\begin{lem} \label{lem22.4} Let \(x\geq 1\) be a large number, and let \(\lambda(n)\) be the Carmichael totient function. Then
\begin{equation}
\sum_{p\leq x }\frac{\lambda(p-1)}{p-1} \gg \frac{x}{(\log x)(\log \log x)}.
\end{equation}
\end{lem}

\begin{proof} Begin with the expression
\begin{eqnarray}
\frac{\lambda(p-1))}{p-1} &=&\prod_{q \mid \lambda(p-1)} \left (1- \frac{1}{q} \right ) \nonumber \\
&\gg& \frac{1}{\log \lambda(p-1)}  
\end{eqnarray}
holds for all integers $p-1 \geq 1$, with equality on a subset of integers of zero density. In light of the lower bound and upper bound
\begin{equation}
\frac{n}{(\log n )^{a\log \log \log n}} \leq \lambda(n) \leq \frac{n}{(\log n )^{b\log \log \log n}}
\end{equation} 
for some constants $a>0$ and $b>0$ and large $n \geq 1$, see \cite[Theorem 1]{ES91},  
proceed to determine a lower bound for the sum: 
\begin{eqnarray}
\sum_{p\leq x }\frac{\lambda(p-1)}{p-1} &\gg & k_1\int_2^x \frac{1}{\log \log t} d \pi(t) \nonumber \\
&\gg& \frac{x}{(\log x)(\log \log x)}.
\end{eqnarray}
Confer the exercises for similar information.
\end{proof}

%bbbbbbbbbbbbbbbbbbbbbbbbbbbbbbbbbbbbbbbbbbbb
\subsection{Sums Of Totients Functionsc Over Subsets Of Integers} \label{sec6}
The asymptotic formulas for the normalized summatory totient function $\varphi(n)/n$ and $\varphi(n)$ over the subset of integers $\mathcal{A}=\{n\geq 1:\gcd(\varphi(n),q)=1\}$ are computed here. These results are based on the counting function $A(x)=\{n \leq x:n \in \mathcal{A}\}$.

\begin{thm}  \label{thm6.1} {\normalfont (\cite[Theorem 2]{MH05})} For a prime power $q \geq 2$, and a large number \(x\geq 1,\) the counting function $A(x)$ has the asymptotic foirmula
\begin{equation}
\sum_{\substack{n\leq x\\ \gcd(\varphi(n),q)=1}} 1=c_q\frac{x}{\log^{1/(q-1)} x} \left (1+O_q\left(\frac{\log \log x}{\log x} \right ) \right ),
\end{equation}
where $c_q>0$ is a constant.
\end{thm}
\begin{thm}  \label{thm6.2} For large number \(x\geq 1,\) the average order for the normalized Euler totient function \(\varphi
	(n)/n\) over the subset $\mathcal{A}$ has the asymptotic formula
\begin{equation}
\sum_{\substack{n\leq x\\ \gcd(\varphi(n),q)=1}}\frac{\varphi(n)}{n}=\frac{6c_q}{\pi^2}\frac{x}{\log^{1/(q-1)} x} \left (1+O_q\left(\frac{\log \log x}{\log x} \right ) \right ),
\end{equation}
where $c_q>0$ is a constant.
\end{thm}

\begin{proof} Let $\mathcal{A}=\{n\geq 1:\gcd(\varphi(n),q)=1\}$ and let $A(x)=\{n \leq x:n \in \mathcal{A}\}$ be the corresponding the counting function. Using the standard indentity $\varphi(n)=n\sum_{d \mid n} \mu(n)/d$ the average order is expressed as
\begin{eqnarray}
\sum_{\substack{n\leq x\\ \gcd(\varphi(n),q)=1}} \frac{\varphi(n)}{n}
&=& \sum_{\substack{n\leq x\\ \gcd(\varphi(n),q)=1}}\sum_{d \mid n} \frac{\mu(n)}{d}\nonumber\\
&=&\sum_{d\leq x}\frac{\mu(n)}{d}\sum_{\substack{n\leq x\\ \gcd(\varphi(n),q)=1 \\d \mid n}} 1\\
&=&\sum_{d\leq x}\frac{\mu(n)}{d}\sum_{\substack{n\leq x/d\\ \gcd(\varphi(n),q)=1 }} 1 \nonumber, 
\end{eqnarray}
where \(\mu (n)\in \{ -1, 0, 1 \}\) is the Mobius function. Applying Theorem \ref{thm6.1} leads to
\begin{eqnarray}
\sum_{\substack{n\leq x\\ \gcd(\varphi(n),q)=1}} \frac{\varphi(n)}{n}
&=& \sum_{d\leq x} \frac{\mu(d)}{d} \left (c_q\frac{x/d}{\log^{1/(q-1)} x/d} \left (1+O_q\left(\frac{\log \log x/d}{\log x/d} \right ) \right ) \right ) \\
&=& c_q\frac{x}{\log^{1/(q-1)} x} \left (1+O_q\left(\frac{\log \log x}{\log x} \right ) \right )\sum_{d\leq x} \frac{\mu(d)}{d^2} \nonumber,
\end{eqnarray}
where the implied constant absorbs a negligible dependence on $d$. Now, use Lemma \ref{lem21040} to approximate the finite sum as
\begin{equation}
\sum _{n \leq  x} \frac{\mu(n)}{n^2}=\frac{6}{\pi ^2}+O\left(\frac{1}{x \log ^2 x}\right) 
\end{equation}
and to complete the proof.
\end{proof}

\begin{thm}  \label{thm6.3} For large number \(x\geq 1,\) the average order of the Euler totient function \(\varphi
	(n)\) over the subset $\mathcal{A}=\{n\geq 1:\gcd(\varphi(n),q)=1\}$ has the asymptotic formula
\begin{equation}
\sum_{\substack{n\leq x\\ \gcd(\varphi(n),q)=1}}\varphi(n)=\frac{3c_q}{\pi^2}\frac{x^2}{\log^{1/(q-1)} x} \left (1+O_q\left(\frac{\log \log x}{\log x} \right ) \right ),
\end{equation}
where $c_q>0$ is a constant.
\end{thm}

\begin{proof} By Theorem \ref{thm6.2}, the appropiate measure is \(W(x)=\sum _{n \leq  x, \gcd(\varphi(n),q)=1} \varphi (n)/n\), and summation by part yields
\begin{eqnarray}
\sum_{\substack{n\leq x\\ \gcd(\varphi(n),q)=1}}\varphi(n)&=&\sum_{\substack{n\leq x\\ \gcd(\varphi(n),q)=1}} n\cdot\frac{\varphi(n)}{n}\nonumber\\
&=&\int_{1}^{x}t\: d W(t)\nonumber\\
&=& xW(x)+O(1)-\int_{1}^{x} W(t)d t\\
&=& x \left (\frac{6c_q}{\pi^2}\frac{x}{\log^{1/(q-1)} x} \left (1+O_q\left(\frac{\log \log x}{\log x} \right ) \right ) \right)-\int_{1}^{x} W(t) d t \nonumber\\
&=&\frac{3c_q}{\pi^2}\frac{x}{\log^{1/(q-1)} x} \left (1+O_q\left(\frac{\log \log x}{\log x} \right ) \right )\nonumber .
\end{eqnarray}
     
\end{proof}

%ppppppppppppppppppppppppppppppppppppppppppppppppppppppppppppppppppppppppppp
\subsection{Problems}
\begin{enumerate}
\item Prove that the average order of the Euler totient function over the shifted primes satisfies
\begin{equation}
\sum_{p\leq x }\frac{\varphi(p-a)}{p-a} =a_1 \li(x) +o \left (\li(x)\right )\nonumber,
\end{equation}
$\li(x)$ is the logarithm function, and $a_1>0 $ is a constant depending on $a \geq 1$.\\

\item Prove that the average order of the Carmichael totient function over the shifted primes satisfies
\begin{equation}
\sum_{p\leq x }\frac{\lambda(p-a)}{p-a} =b_1 \frac{\li(x)}{\log \log x} +o \left (\frac{\li(x)}{\log \log x}\right )\nonumber.
\end{equation}
where $b_1>0 $ is a constant depending on $a \geq 1$.\\

\item Estimate the normal order of the totient ratio $\xi(n)=\varphi(n)/\lambda(n)$: For any number $\varepsilon>0$, there exists a function $f(n)$ such that 
\begin{equation}
f(n)-\varepsilon \leq \xi(n) \leq f(n)+\varepsilon \nonumber.
\end{equation}

\item Estimate the average order of the totient ratio $\xi(n)=\varphi(n)/\lambda(n)$ over the integers and over the shifted primes:
\begin{equation}
\sum_{n\leq x }\xi(n)=\sum_{n\leq x }\frac{\varphi(n)}{\lambda(n)} \qquad \text{ and } \qquad \sum_{p\leq x }\xi(p-1)=\sum_{p\leq x }\frac{\varphi(p-1)}{\lambda(p-1)} \nonumber.
\end{equation}
\item Estimate the average order of the inverse totient function over the integers and over the shifted primes:
\begin{equation}
\sum_{n\leq x }\frac{1}{\varphi(n)} \qquad \text{ and } \qquad \sum_{p\leq x }\frac{1}{\varphi(p-1)} \nonumber.
\end{equation}
\item Estimate the average order of the inverse lambda function over the integers and over the shifted primes:
\begin{equation}
\sum_{n\leq x }\frac{1}{\lambda(n)} \qquad \text{ and } \qquad \sum_{p\leq x }\frac{1}{\lambda(p-1)} \nonumber.
\end{equation}

\end{enumerate}

%3333333333333333333333333333333333333333333333333333
\newpage
\section{Finite Cyclic Groups} \label{s3}
Let $n=p_1^{v_1} p_2^{v_2} \cdots p_t^{v_t}$ be an arbitrary integer, and let $\Z/n\Z$ be a finite ring. Some properties of the group of units of the finite ring 
\begin{equation}
\left (\mathbb{Z} /n \mathbb{Z}\right )^{\times}=U(p^{v_1}) \times U(p^{v_2}) \times  \cdots \times U(p^{v_t}), 
\end{equation}
where $U(n)$ is a cyclic group of order $\#U(n)=\varphi(n)$ are investigated here.\\

\subsection{Multiplicative Orders}
\begin{dfn} { \normalfont The \textit{order} of an element $v \in G$ in a cyclic group $G$ is defined by $\ord_G (v)=\min \{n:v^n\equiv 1 \bmod G\}$, and the index is defined by $\ind_G (v)=\#G/\ord_G (v)$.}
\end{dfn}

\begin{dfn} { \normalfont A subset of integers $\mathcal{B} \subset \Z$ with respect to a fixed base $v \geq 2$ if the order and the index are nearly equal: $\ord_n (v) \approx  \ind_n (v)\approx \sqrt{n}$ for each $n \in \mathcal{B} $.}
\end{dfn}

\begin{lem} \label{lem3.45} The order $\ord : G \longrightarrow \N$ is multiplicative function on a multiplicative subgroup $G$ of cardinality $\#G=\lambda(n)$ in $\Z/n\Z$, and it has the followings properties.\\

\begin{tabular}{*{2}{p{0.5\textwidth}}}
{\normalfont (i)}      	$\ord_n(u\cdot v)=\ord_n(u)\ord_n(v)$, & if $\gcd(\ord_n(u),\ord_n(v))=1.$   \\
{\normalfont (ii)}		 $\ord_n(u^k)=\ord_n(u)/\gcd(k,n)$, &  for any pair of integers $k, n \geq 1.$
	\end{tabular}  
\end{lem} 

The Carmichael function specifies the maximal order of a cyclic subgroup $G$ of the finite ring $\mathbb{Z}/n\mathbb{Z}$, and the maximal order $\lambda(n)= \max \{ m\geq 1:v^m \equiv 1 \bmod n \}$ of the elements in a finite cyclic group $G$.\\

\begin{dfn} { \normalfont An integer \(u\in \mathbb{Z}\) is called a \textit{primitive root} \(\text{mod } n \) if the least exponent \(\min  \left\{ m\in \mathbb{N}:u^m\equiv 1 \text{ mod }n \right\}=\lambda
(n)\). }
\end{dfn}
 In synopsis, primitive elements in a cyclic group have the maximal orders $\ord_G (v)=\#G$, and minimal indices $\ind_G (v)=1$.
\begin{lem} \label{lem3.4} { \normalfont (Primitive root test)} Let \(p\geq 3\) be a prime, and let \(a \geq 2\) be an integer such $\gcd(a,p)=1$. Then, 

the integer $a$ is a primitive root if and only if 
\begin{equation}
a^{(p-1)/q} -1 \not \equiv 0 \bmod p
\end{equation}
for every prime divisor $q \mid p-1$.
\end{lem}
\begin{proof} This is a restricted version of the Pocklington primality test, see \cite[p.\ 175]{CP05}. 
\end{proof}

\begin{lem} \label{lem3.1}  For any integer $n \geq 1$, and the group $\mathbb{Z} /n \mathbb{Z}$, the followings hold. 
\begin{enumerate} [font=\normalfont, label=(\roman*)]
\item The group of units $\left (\mathbb{Z} /n \mathbb{Z}\right )^{\times}$ has $\varphi(n)$ units.
\item The number of primitive root is given by
\begin{equation}
\varphi(\varphi(n))=\varphi(n)\prod_{p \mid \varphi(n)} \left (1-\frac{1}{p^{m(p)}}\right ), 
\end{equation}
where $m(p)\geq 1$ is the number of invariant factor associate with $p$.
\end{enumerate} 
\end{lem} 

The Euler totient function and the more general Carmichael totient function over the finite ring $\mathbb{Z}/n\mathbb{Z}$ are seamlessly
linked by the Fermat-Euler Theorem. 
\begin{lem} \label{lem3.2} {\normalfont (Fermat-Euler)}  If \(a\in \mathbb{Z}\) is an integer such that \(\gcd (a,n)=1,\) then \(a^{\varphi (n)}\equiv
	1 \bmod n\).
\end{lem} 
The improvement provides the least exponent \(\lambda (n) \mid \varphi (n)\) such that \(a^{\lambda (n)}\equiv
(1 \mod n)\). \\

\begin{lem} \label{lem3.3} { \normalfont (\cite{CR10})}  Let \(n\in \mathbb{N}\) be any given integer. Then
\begin{enumerate} [font=\normalfont, label=(\roman*)]
\item The congruence \(a^{\lambda (n)}\equiv
1 \bmod n\) is satisfied by every integer \(a\geq 1\) relatively prime to \(n\), that is \(\gcd (a,n)=1\).

\item In every congruence \(x^{\lambda (n)}\equiv 1 \bmod n\), a solution \(x=u\) exists which is a primitive root \(\bmod  n\), and for any such
solution \(u\), there are \(\varphi (\lambda (n))\) primitive roots congruent to powers of \(u\).
\end{enumerate} 
\end{lem}

\begin{proof} (i) The number \(\lambda (n)\) is a multiple of every \(\lambda \left(p^v\right)=\varphi\left(p^v\right) \) such that \(p^v \mid  n\). 

Ergo, for any relatively prime integer \(a\geq 2\), the system of congruences 
\begin{equation}
a^{\lambda (n)}\equiv 1\bmod p_1^{v_1}, \quad a^{\lambda (n)}\equiv 1\bmod p_2^{v_2}, \quad \ldots, \quad a^{\lambda (n)}\equiv 1\bmod p_t^{v_t},
\end{equation}
where \(t=\omega (n)\) is the number of prime divisors in \(n\), is valid. 
\end{proof}

\subsection{Maximal Cyclic Subgroups}
The multiplicative group \( \left ( \mathbb{Z} / n \mathbb{Z} \right )^{\times} \) has $\xi(n)=\varphi(n)/\lambda(n)$ maximal cyclic subgroups 
\begin{equation}
G_1 \cup G_2 \cup \cdots \cup  G_t =  \left ( \mathbb{Z} / n \mathbb{Z} \right )^{\times} 
\end{equation}
of order $\#G_i= \lambda(n)$, and $G_i \cap G_j=\{1\} $ for $i \ne j$ with $1 \leq i ,j\leq t=\xi(n)$. Each maximal subgroup $G_i$ has a unique subset 

of $\varphi(\lambda(n))$ primitive roots. The optimal case \( G_1=\left ( \mathbb{Z} / n \mathbb{Z} \right )^{\times} \) for $\xi(n)=1$ occurs on a 

subset of integers of zero density, the next lemma is the best known result, see also Lemma \ref{lem22.3}.

\begin{lem} \label{lem3.5} {\normalfont (Gauss)} 
Let \(p\geq 3\) be a prime, and let \(n \geq 1\) be an integer. Then, the multiplicative groups has the following properties.\\

\begin{enumerate}[font=\normalfont, label=(\roman*)]
\item  $\left ( \mathbb{Z} / p^n \mathbb{Z} \right )^{\times} $ is cyclic of order $\varphi(p^n)$, and there exists a primitive root of the same order.
\item  $\left ( \mathbb{Z} / 2p^n \mathbb{Z} \right )^{\times} $ is cyclic of order $\varphi(2p^n)$, and there exists a primitive root of the same order.	
\end{enumerate}
\end{lem}
\begin{proof} The proof and additional information appear in \cite[Theorem 10.7]{AP76}, and \cite[Theorem 2.6]{RH95}. 
\end{proof}
Extensive details on this topic appear in \cite{CP09}.

%33-33-33-33-33-33-33-33-33-33-33-33-33-33-33-33-33-33-33-33-33-33-33-33-33-33
\newpage
\section{Characteristic Functions} \label{s33}
The indicator function or characteristic function  \(\Psi :G\longrightarrow \{ 0, 1 \}\) of some distinguished subsets of elements are not difficult to construct.  Many equivalent representations of the  characteristic function $\Psi $ of the elements are possible. \\

%ssssssssssssssssssssssssssssssssssssssssssssssssssssssssssssssssssssssssss
\subsection{Characteristic Functions Modulo Prime Powers}
The standard method for constructing characteristic function for primitive elements are discussed in \cite[Corollary 3.5]{RH95}, \cite[p.\, 258]{LN97}, and some characteristic functions for finite rings are discussed in \cite{JJ71}. These type of characteristic functions detect the orders of the elements \(v\in \left ( \mathbb{Z} / p^2 \mathbb{Z} \right )^{\times} \) by means of the divisors of $\varphi(p^2)=\#G $.  A new method for constructing characteristic functions for certain elements in cyclic groups is developed here. These type of characteristic functions detect the orders of the elements \(v\in \left ( \mathbb{Z} / p^2 \mathbb{Z} \right )^{\times} \) by means of the solutions of the equation $\tau ^{pn}-v \equiv 0 \bmod p^2 $, where
\(v,\tau\) are constants, and $n$ is a variable such that \(1\leq n<p-1$, and $\gcd (n,p-1)=1\). The formula $\varphi(n)= \prod_{p \mid n}(1-1/p)$ denotes the Euler totient function.\\ 

\begin{lem} \label{lem33.5}
Let \(p\geq 3\) be a prime, and let \(\tau\) be a primitive root $\bmod \, p^2$. Let \(v\in
\left ( \mathbb{Z} / p^2 \mathbb{Z} \right )^{\times} \) be a nonzero element. Then
\begin{equation}
\Psi_v (p^2)=\sum _{\substack{1 \leq n <p-1 \\\gcd (n,p-1)=1}} \frac{1}{\varphi(p^2)}\sum_{0\leq m< \varphi(p^2)} e^{\frac{i 2\pi  (\tau ^{pn}-v)m}{\varphi(p^2)} }
=\left \{\begin{array}{ll}
1 & \text{  if } \ord_{p^2} (v)=p-1,  \\
0 & \text{ if } \ord_{p^2} (v)\neq p-1. \\
\end{array} \right .
\end{equation}
\end{lem}

\begin{proof} Let $ \tau \in \left ( \mathbb{Z} / p^2 \mathbb{Z} \right )^{\times} $  be a fixed primitive root of order $p(p-1)=\varphi(p^2)$. As the 

index \(n\geq 1\) ranges over the integers relatively prime to \(p-1\), the element \(\tau ^{pn}\in \left ( \mathbb{Z} / p^2 \mathbb{Z} \right )^{\times} \) ranges over the elements of order $ \ord_{p^2} (\tau^{pn})=p-1$. Hence, the equation 
\begin{equation}
\tau ^{pn}-v=0
\end{equation} has a solution if and only if the fixed element \(\left ( \mathbb{Z} / p^2 \mathbb{Z} \right )^{\times}\) is an elements of 
order $\ord_{p^2} (v)=p-1$. Setting \(w=e^{i 2\pi  (\tau ^{pn}-v)/\varphi(p^2)}\) and summing the inner sum yield 
\begin{equation}
\sum_{\gcd (n,p-1)=1} \frac{1}{\varphi(p^2)}\sum_{0\leq m <\varphi(p^2)} w^m=
\left \{
\begin{array}{ll}
1 & \text{ if } \ord_{p^2} (v)=p-1,  \\
0 & \text{ if } \ord_{p^2} (v)\neq p-1. \\
\end{array} \right .
\end{equation} 

This follows from the geometric series identity $\sum_{0\leq m\leq x-1} w^m =(w^x-1)/(w-1),w\ne 1$ applied to the inner sum. 
\end{proof}

The characteristic function for any element $v \geq 2$ of order \(\ord_{p^2}(v) =d \mid p-1\) in the cyclic group \( \left ( \mathbb{Z} / p^2 \mathbb{Z} \right )^{\times} \) is a sum of characteristic functions.

\begin{lem} \label{lem33.12} Let Let $ v\geq 2$ be a fixed base, let \(p\geq 3\) be a prime, and let \(\tau\) be a primitive root $\bmod \, p^2$. The indicator function for the subset of primes such that $v^{p-1}-1 \equiv 0 \bmod p^2$ is given by 
\begin{eqnarray} \label{33.200}
\Psi_{0} (p^2)&=&\sum_{d \mid p-1}\sum _{\substack{1 \leq n <p-1 \\ \gcd (n,(p-1)/d)=1}} \frac{1}{\varphi(p^2)}\sum_{0\leq m< \varphi(p^2)} e^{\frac{i 2\pi  (\tau ^{dpn}-v)m}{\varphi(p^2)} }\\
&=&\left \{\begin{array}{ll}
1 & \text{ if } \ord_{p^2} (v) \mid p-1,  \\
0 & \text{ if } \ord_{p^2} (v)\nmid p-1. \\
\end{array} \right . \nonumber
\end{eqnarray}
\end{lem}

\begin{proof} Suppose that $\ord_{p^2} (v)=p-1$. Then, there is a unique pair $d \mid p-1$ and $n\geq 1$ with $\gcd(n,(p-1)/d)=1$ such that $
\tau ^{dpn}-v\equiv0 \bmod p^2.$
Otherwise, $\tau ^{dpn}-v\not \equiv 0 \bmod p^2$ for all pairs $d\mid p-1$ and $\gcd(n, (p-1)/d)=1$. Proceed as in the proof of Lemma \ref{lem33.5}.
\end{proof}

\begin{lem} \label{lem33.7}
Let \(p\geq 3\) be a prime, and let \(\tau\) be a primitive root $\bmod \, p^k$. Let \(v\in
\left ( \mathbb{Z} / p^k \mathbb{Z} \right )^{\times} \) be a nonzero element. Then
\begin{equation}
\Psi_v (p^k)=\sum _{\substack{1 \leq n <p-1 \\\gcd (n,p-1)=1}} \frac{1}{\varphi(p^k)}\sum_{0\leq m< \varphi(p^k)} e^{\frac{i 2\pi  (\tau ^{p^{k-1}n}-v)m}{\varphi(p^k)} }
=\left \{\begin{array}{ll}1 & \text{ if } \ord_{p^k} (v)=p-1,  \\
0 & \text{ if } \ord_{p^k} (v)\neq p-1. \\
\end{array} \right .
\end{equation}
\end{lem}

\begin{proof} Modify the proof of Lemma \ref{lem33.5} to fit the finite ring $\Z/p^k \Z$. \end{proof}

%ssssssssssssssssssssssssssssssssssssssssssssssssssssssssssssssssssssssssssssssssssss
\subsection{Characteristic Functions Modulo $n$}
The indicator function for primitive root in a maximal cyclic group $G \subset \Z/n \Z$ is simpler than the indicator function for primitive root in $\Z/n \Z$, which is a sum of indicator functions for its maximal cyclic groups $G_1, G_2, \ldots , G_e$, with $e\geq1$.

\begin{lem} \label{lem33.9}
Let \(n\geq 3\) be an integer, and let \(\tau \in G\) be a primitive root $\bmod \, n$ in a maximal cyclic subgroup \(G \subset
\left ( \mathbb{Z} / n \mathbb{Z} \right )^{\times} \). If $v\ne \pm u^2$ is an integer, then
\begin{equation}
\Psi_v(G)=\sum _{\substack{1 \leq m <\lambda(n) \\ \gcd (m,\lambda(n))=1}} \frac{1}{\varphi(n)}\sum_{0\leq r< \varphi(n)} e^{\frac{i 2\pi  (\tau ^{m}-v)r}{\varphi(n)} }
=\left \{
\begin{array}{ll}
1 & \text{ if } \ord_{n} (v)=\lambda(n),  \\
0 & \text{ if } \ord_{n} (v)\neq \lambda(n). \\
\end{array} \right .
\end{equation}
\end{lem}

\begin{proof} Let $ \tau \in G $  be a fixed primitive root of order $\lambda(n)$, see Lemma \ref{lem3.5}. As the index \(m\geq 1\) ranges over the integers relatively prime to \(\lambda(n)\), the element \(\tau ^{m}\in G \) ranges over the elements of order $ \ord_{n} (\tau^{m})=\lambda(n)$. Hence, the equation 
\begin{equation}
\tau ^{m}-v=0
\end{equation} has a solution if and only if the fixed element \(v \in G\) is an element of order $\ord_{n} (v)=\lambda(n)$. Next, let \(w=e^{i 2\pi  (\tau ^{m}-v)/\varphi(n)}\). Summing the inner sum yields 
\begin{equation}
\sum_{\gcd (m,\lambda(n))=1} \frac{1}{\varphi(n)}\sum_{0\leq r <\varphi(n)} e^{\frac{i 2\pi  (\tau ^{m}-v)r}{\varphi(n)} }=
\left \{
\begin{array}{ll}
1 & \text{ if } \ord_{n} (v)=\lambda(n),  \\
0 & \text{ if } \ord_{n} (v)\neq \lambda(n). \\
\end{array} \right .
\end{equation} 

This follows from the geometric series identity $\sum_{0\leq n\leq x-1} w^n =(w^x-1)/(w-1),w\ne 1$ applied to the inner sum. 
\end{proof}

\begin{lem} \label{lem33.17}
Let \(n\geq 3\) be an integer, and let $\xi(n)= \varphi(n)/\lambda(n)$. Let \(\tau_i \in G_i\) be a primitive root $\bmod  n$ in a maximal cyclic subgroup \(G_i \subset
\left ( \mathbb{Z} / n \mathbb{Z} \right )^{\times} \). If $v\ne \pm u^2$ is an integer, then
\begin{equation}
\Psi_1(n)=\sum_{1 \leq i \leq \xi(n)} \sum _{\substack{1 \leq m <\lambda(n) \\ \gcd (m,\lambda(n))=1}} \frac{1}{\varphi(n)}\sum_{0\leq r< \varphi(n)} e^{\frac{i 2\pi  (\tau_i ^{m}-v)r}{\varphi(n)} }
=\left \{
\begin{array}{ll}
1 & \text{ if } \ord_{n} (v)=\lambda(n),  \\
0 & \text{ if } \ord_{n} (v)\neq \lambda(n). \\
\end{array} \right .
\end{equation}
\end{lem}
\begin{proof} This is a sum of $\xi(n)\geq 1$ copies of the indicator function proved in Lemma \ref{lem33.9}.
\end{proof}

The last one considered is the indicator function for elements of order $\ord_{n^2}(v) \mid \lambda(n)$ in $\Z/n^2 \Z$. This amounts to a double sum of indicator functions for its maximal cyclic groups $G_1, G_2, \ldots , G_e$, with $e\geq1$.

\begin{lem} \label{lem33.19}
Let \(n\geq 3\) be an integer, and let $\xi(n)= \varphi(n)/\lambda(n)$. Let \(\tau_i \in G_i\) be a primitive root $\bmod  n$ in a maximal cyclic 

subgroup \(G_i \subset
\left ( \mathbb{Z} / n^2 \mathbb{Z} \right )^{\times} \). If $v\ne \pm u^2$ is an integer, then
\begin{eqnarray}
\Psi_0(n^2)&=&\sum_{1 \leq i \leq \xi(n)} \sum_{d \mid \lambda(n)}\sum _{\substack{1 \leq m <\lambda(n) \\ \gcd (m,\lambda(n)/d)=1}} \frac{1}{\varphi(n^2)}\sum_{0\leq r< \varphi(n^2)} e^{\frac{i 2\pi  (\tau_i ^{dm}-v)r}{\varphi(n^2)} } \nonumber \\
&=&\left \{
\begin{array}{ll}
1 & \text{ if } \ord_{n^2} (v) \mid\lambda(n),  \\
0 & \text{ if } \ord_{n^2} (v)\nmid \lambda(n). \\
\end{array} \right .
\end{eqnarray}
\end{lem}
\begin{proof} For each divisor $d \mid \lambda(n)$, this is a sum of $\xi(n)\geq 1$ copies of the indicator function proved in Lemma \ref{lem33.9}.
\end{proof}

%ppppppppppppppppppppppppppppppppppppppppppppppppppppppp
\subsection{Problems}
\begin{enumerate}
\item Let $p \geq 3$ be a prime. Show that the characteristic function of quadratic nonresidue in the finite ring $ \left ( \mathbb{Z} / p^2 \mathbb{Z} \right )^{\times} $ is 
$$
\Psi_{v^2} (p^2)=\sum _{\substack{1 \leq n <\varphi(p^2) \\\gcd (2,n)=1}} \frac{1}{\varphi(p^2)}\sum_{0\leq m< \varphi(p^2)} e^{\frac{i 2\pi  (\tau ^{\frac{p(p-1)}{2}n}-v)m}{\varphi(p^2)} }
=\left \{
\begin{array}{ll}
1 & \text{ if } \ord_{p^2} (v) =2,  \\
0 & \text{ if } \ord_{p^2} (v)\ne 2. \\
\end{array} \right .
$$
\item Let $p =3a+1\geq 7$ be a prime. Show that the characteristic function of cubic nonresidue in the finite ring $ \left ( \mathbb{Z} / p^2 \mathbb{Z} \right )^{\times} $ is 
$$
\Psi_{v^3} (p^2)=\sum _{\substack{1 \leq n <\varphi(p^2) \\\gcd (3,n)=1}} \frac{1}{\varphi(p^2)}\sum_{0\leq m< \varphi(p^2)} e^{\frac{i 2\pi  (\tau ^{\frac{p(p-1)}{3}n}-v)m}{\varphi(p^2)} }
=\left \{
\begin{array}{ll}
1 & \text{ if } \ord_{p^2} (v) =3,  \\
0 & \text{ if } \ord_{p^2} (v)\ne 3. \\
\end{array} \right .
$$
\end{enumerate}

%sssssssssssssssssssssssssssssssssssssssssssssssssssssssssssssssss
\section{Equivalent Exponential Sums} 
For any fixed $ 0 \ne s \in \mathbb{Z}/p^2\Z$, an asymptotic relation for the exponetial sums 
\begin{equation} \label{330}
\sum_{\gcd(n,\varphi(p^2))=1} e^{i2\pi s \tau^n/\varphi(p^2)} \quad \text{ and } \quad \sum_{\gcd(n,\varphi(p^2))=1} e^{i2\pi \tau^n/\varphi(p^2)},
\end{equation}
is provided in Lemma \ref{lem33.22}. This result expresses the first exponential sum in (\ref{330}) as a sum of simpler exponential sum and an error term. The proof is based on Lagrange resolvent in the finite ring $ \mathbb{Z}/p^2\Z$. Specifically,
\begin{equation} \label{344}
(\omega^t,\zeta^{s\tau})=\zeta^s+\omega^{-t} \zeta^{s\tau}+\omega^{-2t} \zeta^{s\tau^{2}}+ \cdots +\omega^{-(p-2)t}\zeta^{s\tau^{\varphi(p^2)-1}}, 
\end{equation}
where $\omega=e^{i 2 \pi/p}$, $\zeta=e^{i 2 \pi/\varphi(p^2)}$, and  the variables $0 \ne s \in \mathbb{Z}/p^2 \Z $, and $0 \ne t \in \mathbb{Z}/p \Z$.

\begin{lem}   \label{lem33.22}  Let \(p\geq 2\) be a large prime. If $\tau $ be a primitive root modulo $p^2$, then,
	\begin{equation} 
	\sum_{\gcd(n, (p-1)/d)=1} e^{i2\pi s \tau^{ndp}/\varphi(p^2)} = \sum_{\gcd(n, (p-1)/d)=1} e^{i2\pi  \tau^{ndp}/\varphi(p^2)} + O(p^{1/2} \log^3 p),
	\end{equation} 
	for any fixed $d \mid p-1$, and $ 0 \ne s \in \mathbb{Z}/p^2 \Z$. 	
\end{lem} 
\begin{proof} Summing (\ref{344}) times $\omega^{tn}$ over the variable $t \in \mathbb{Z}/p \Z$ yields, (all the nontrivial complex $p$th root of unity),
	\begin{equation} \label{366}
	p \cdot 	e^{i2\pi s \tau^{ndp}/\varphi(p^2)} = \sum_{0 \leq t \leq p-1} (\omega^t, \zeta^{s\tau^{ndp}})\omega^{tn}.
	\end{equation} 
	Summing (\ref{366}) over the variable $n\geq 1$, for which  $\gcd(n, (p-1)/d)=1$, yields 
	\begin{eqnarray} \label{320}
	p \cdot 	\sum_{\gcd(n, (p-1)/d)=1} e^{i2\pi s \tau^{ndp}/\varphi(p^2)} &=& \sum_{\gcd(n, (p-1)/d)=1,}  \sum_{0 \leq t \leq p-1}(\omega^t, \zeta^{s\tau^{dp}})\omega^{tn} \\
	&=& \sum_{1\leq t \leq p-1} (\omega^t, \zeta^{s\tau^{dp}}) \sum_{\gcd(n, (p-1)/d)=1}  \omega^{tn}-p \nonumber.
	\end{eqnarray} 
	The first index $t=0$ contributes $p$, see \cite[Equation (5)]{ML72} for similar calculations. Likewise, the basic exponential sum for $s=1$ can be written as
	\begin{equation} \label{321}
	p \cdot 	\sum_{\gcd(n, (p-1)/d)=1} e^{i2\pi  \tau^{ndp}/\varphi(p^2)} = \sum_{1 \leq t \leq p-1} (\omega^t, \zeta^{\tau^{dp}})\sum_{\gcd(n, (p-1)/d)=1}  \omega^{tn}-p,
	\end{equation} 
	Differencing (\ref{320}) and (\ref{321}) produces 
	\begin{eqnarray} \label{390}
	S_1&= &p \cdot\left (	\sum_{\gcd(n, (p-1)/d)=1} e^{i2\pi s \tau^{ndp}/\varphi(p^2)} -\sum_{\gcd(n, (p-1)/d)=1} e^{i2\pi  \tau^{ndp}/\varphi(p^2)}\right ) \nonumber\\
	&=&  \sum_{1 \leq t \leq p-1}   \left ((\omega^t,\zeta^{s\tau^{dp}})-(\omega^t,\zeta^{\tau^{dp}}) \right) \sum_{\gcd(n, (p-1)/d)=1}\omega^{tn} .
	\end{eqnarray}
	The right side sum $S_1$ can be rewritten as 
	\begin{eqnarray} \label{397}
	S_1
	& = &    \sum_{1 \leq t \leq p-1}   \left ( (\omega^t,\zeta^{s\tau^{ndp}})-(\omega^t,\zeta^{\tau^{ndp}}) \right)  \sum_{\gcd(n, (p-1)/d)=1}\omega^{tn}  \\
	& = &    \sum_{1 \leq t \leq p-1}    \left ( (\omega^t,\zeta^{s\tau^{ndp}})-(\omega^t,\zeta^{\tau^{ndp}}) \right) 
	\sum_{e \leq (p-1)/d} \mu(e) \frac{\omega^{et}-\omega^{et(\frac{p-1}{d}+1)}}{1-\omega^{et}} \nonumber \\
	& = &    \sum_{1 \leq t \leq p-1,} \sum_{e \leq (p-1)/d}     \left ( (\omega^t,\zeta^{s\tau^{ndp}})-(\omega^t,\zeta^{\tau^{ndp}}) \right) 
	\mu(e) \frac{\omega^{et}-\omega^{et(\frac{p-1}{d}+1)}}{1-\omega^{et}} \nonumber .
	\end{eqnarray}
	The second line follows from Lemma \ref{lem33.3}-i. The upper bound
	\begin{eqnarray} \label{398}
	|S_1|
	& \leq &    \sum_{1 \leq t \leq p-1,} \sum_{e \leq (p-1)/d}  \left|     \left ( (\omega^t,\zeta^{s\tau^{ndp}})-(\omega^t,\zeta^{\tau^{ndp}}) \right) 
	\mu(e) \frac{\omega^{et}-\omega^{et(\frac{p-1}{d}+1)}}{1-\omega^{et}} \right | \nonumber \\
    & \leq &    \sum_{1 \leq t \leq p-1,} \sum_{e \leq (p-1)/d}  \left|      (\omega^t,\zeta^{s\tau^{ndp}})-(\omega^t,\zeta^{\tau^{ndp}})  \right | \left |
	\mu(e) \frac{\omega^{et}-\omega^{et(\frac{p-1}{d}+1)}}{1-\omega^{et}} \right | \nonumber \\
		& \leq &    \sum_{1 \leq t \leq p-1,} \sum_{e \leq (p-1)/d}      \left ( 2p^{1/2} \log p \right)  \cdot \left |
	\mu(e) \frac{\omega^{et}-\omega^{et(\frac{p-1}{d}+1)}}{1-\omega^{et}} \right | \nonumber \\
 	& \leq &   \sum_{1 \leq t \leq p-1}  \left (2 p^{1/2} \log p \right) \cdot \left (  \frac{2p \log p}{\pi t} \right )  \nonumber \\
	& \leq &    \left (4p^{3/2 } \log^2 p \right ) \sum_{1 \leq t \leq p-1} \frac{1}{t}  \\ 
	& \leq &    8p^{3/2}  \log^3 p \nonumber .
	\end{eqnarray}
The third line follows the upper bound for Lagrange resolvents, and the fourth line follows from Lemma \ref{lem33.3}-ii. Here, the difference of two Lagrange resolvents, (Gauss sums), has the upper bound
	\begin{equation}
	\left | (\omega^t,\zeta^{s\tau^{dp}})-(\omega^t,\zeta^{\tau^{dp}}) \right | \leq 2\left |\sum_{1 \leq t \leq p-1}    \chi(t) e^{i 2 \pi t/p} \right | \leq 2 p^{1/2} \log p,
	\end{equation}
	where $ \left | \chi(t) \right | =1$ is a root of unity. Taking absolute value in (\ref{390}) and using (\ref{398}) and (\ref{399}) return
	\begin{eqnarray} \label{388}
	p \cdot 	\left|  	\sum_{\gcd(n,(p-1)/d)=1} e^{i2\pi s \tau^n/p} -\sum_{\gcd(n,(p-1)/d)=1} e^{i2\pi  \tau^{ndp}/p} \right|  
	& \leq &    |S_1| \\
	&\leq& 8p^{3/2} \log^3 p \nonumber .
	\end{eqnarray}
	The last inequality implies the claim.
\end{proof}

\begin{lem}   \label{lem33.3}  Let \(p\geq 2\) be a large prime, and let $\omega=e^{i2 \pi/p} $ be a $p$th root of unity. Then,
\begin{enumerate} [font=\normalfont, label=(\roman*)]
\item $$
	 \sum_{\gcd(n,(p-1)/d)=1} \omega^{tn} = \sum_{e \leq (p-1)/d} \mu(e) \frac{\omega^{et}-\omega^{et(\frac{p-1}{d}+1)}}{1-\omega^{et}},
$$

\item $$
	\left | \sum_{\gcd(n,(p-1)/d)=1} \omega^{tn}  \right |\leq \frac{2p \log p}{\pi t},
$$ 
where $\mu(k)$ is the Mobius function, for any fixed pair $d \mid p-1$ and $t \in [1, p-1]$. 
\end{enumerate}
\end{lem} 

\begin{proof} (i) Use the inclusion exclusion principle to rewrite the exponential sum as
	\begin{eqnarray} \label{340}
	\sum_{\gcd(n,(p-1)/d)=1} \omega^{tn}&=& \sum_{n \leq (p-1)/d} \omega^{tn}  \sum_{\substack{e \mid (p-1)/d \\ e \mid n}}\mu(e)  \nonumber \\
	&=& \sum_{e \leq (p-1)/d} \mu(e) \sum_{\substack{n \leq (p-1)/d \\ e \mid n}} \omega^{tn}\nonumber \\
	& =&\sum_{e \leq (p-1)/d} \mu(e) \sum_{m \leq (p-1)/ de} \omega^{etm} \\
	&=& \sum_{e \leq (p-1)/d} \mu(e) \frac{\omega^{et}-\omega^{et(\frac{p-1}{d}+1)}}{1-\omega^{et}} \nonumber.
	\end{eqnarray} 
(ii) Observe that the parameters $\omega=e^{i2 \pi/p}$, the integers $t \in [1, p-1]$, and $e \leq (p-1)/d$ imply that $\pi et/p\ne k \pi $ with $k \in \mathbb{Z}$, so the sine function $\sin(\pi et/p)\ne 0$ is well defined. Using standard manipulations, and $z/2 \leq \sin(z) <z$ for $0<|z|<\pi/2$, the last expression becomes
	\begin{equation}
	\left |\frac{\omega^{et}-\omega^{et(\frac{p-1}{d}+1)}}{1-\omega^{et}} \right |\leq 	\left | \frac{2}{\sin( \pi et/ p)} \right | 
	\leq \frac{2p}{\pi et}
	\end{equation}
	for $1 \leq d \leq p-1$. Finally, the upper bound is
	\begin{eqnarray}
	\left|   \sum_{e \leq (p-1)/d} \mu(e) \frac{\omega^{et}-\omega^{et(\frac{p-1}{d}+1)}}{1-\omega^{et}} \right| 
	&\leq&\frac{2p}{\pi t} \sum_{e \leq (p-1)/d} \frac{1}{e} \\
	&\leq& \frac{2p \log p}{\pi t} \nonumber .
	\end{eqnarray}
\end{proof}

\section{Upper Bound For The Main Term} \label{s44}
An estimate for the finite sum occuring in the evaluation of the main term is considered in this section.
\begin{lem}  \label{lem44.1}  Let \(x\geq 1\) be a large number, and let \(\varphi (n)\) be the Euler totient function. Then
\begin{equation} \label{4095}
\sum_{p\leq x} \frac{1}{\varphi(p^2)}  \sum_{d \mid p-1,} \sum_{\gcd (n,(p-1)/d)=1}1 \leq 2 \log \log x .
\end{equation} 
\end{lem}
\begin{proof} Use the identity $\sum_{d \mid n}\varphi(d)=n$ to eliminate the inner double sum in the following way: 
\begin{equation} \label{4095}
\sum_{d \mid p-1,} \sum_{\gcd (n,(p-1)/d)=1}1 =\sum_{d \mid p-1} \varphi((p-1)/d)=p-1.
\end{equation} 
Substituting this returns
\begin{equation} \label{4095}
\sum_{p\leq x} \frac{1}{\varphi(p^2)}  \sum_{d \mid p-1,} \sum_{\gcd (n,(p-1)/d)=1}1 =\sum_{p\leq x} \frac{1}{\varphi(p^2)}  \cdot (p-1) = \sum_{ p\leq x}  \frac{1}{p}.
\end{equation} 
Lastly, apply Mertens theorem to the prime harmonic sum. 
\end{proof}

%44-44-44-44-44-44-44-44-44-44-44-44-44-44-44-44-44-44-44-44-44-44-44
\newpage
\section{Evaluations Of The Main Terms} \label{s44}
Various types of finite sums occurring in the evaluations of the main terms of various results are considered in this section.\\

\subsection{Sums Over The Primes}
\begin{lem}  \label{lem44.1}  Let \(x\geq 1\) be a large number, and let \(\varphi (n)\) be the Euler totient function. Then
\begin{enumerate} [font=\normalfont, label=(\roman*)]
\item  $ \displaystyle \sum_{p\leq x} \frac{1}{\varphi(p^2)}  \sum_{d \mid p-1,} \sum_{\gcd (n,(p-1)/d)=1}1= \log \log x +b_0  +O \left (\frac{1}{\log x}\right )  ,$ \\
where $b_0>0$ is a constant.
\item $ \displaystyle\sum_{p\leq x} \frac{1}{\varphi(p^2)} \sum_{\gcd (n,p-1)=1}1=a_0 \log \log x +a_1     +O \left (\frac{1}{\log x}\right )  ,$\\
where $a_0>0$ and $a_1$ are constants.
\end{enumerate}
\end{lem}

\begin{proof} (i) Use the identity $\sum_{d \mid n}\varphi(d)=n$ to eliminate the inner double sum in the following way: 
\begin{equation} \label{4095}
\sum_{d \mid p-1,} \sum_{\gcd (n,(p-1)/d)=1}1 =\sum_{d \mid p-1} \varphi((p-1)/d)=p-1.
\end{equation} 
Substituting this returns
\begin{equation} \label{4095}
\sum_{p\leq x} \frac{1}{\varphi(p^2)}  \sum_{d \mid p-1,} \sum_{\gcd (n,(p-1)/d)=1}1 =\sum_{p\leq x} \frac{1}{\varphi(p^2)}  \cdot (p-1) = \sum_{ p\leq x}  \frac{1}{p}.
\end{equation} 
Lastly, apply Lemma \ref{lem22.1} to the prime harmonic sum. (ii) The proof of this case is similar, but it uses Lemma \ref{lem22.3}.
\end{proof}

%ssssssssssssssssssssssssssssssssssssssssssssssssssssssssssssssssssssssssssssssssssss
\subsection{Sums Over The Bases} 
The other form of the main term deals with the summation over the bases $v \geq 2$.\\

\begin{lem}  \label{lem44.2}  Let \(x\geq 1\) be a large number, and let \(\varphi (n)\) be the Euler totient function. Then
\begin{enumerate} [font=\normalfont, label=(\roman*)]
\item $ \displaystyle \sum_{v\leq x} \frac{1}{\varphi(p^2)}\sum_{d \mid n,} \sum_{\gcd(n,(p-1)/d)=1} 1 =\frac{1}{p} x +O\left ( \frac{1}{p}\right ).
$
\item $ \displaystyle 
	\sum_{v\leq x} \frac{1}{\varphi(p^2)}\sum_{\gcd(n,p-1)=1} 1 =\frac{\varphi(p-1)}{\varphi(p^2)} x +O\left ( \frac{1}{p}\right ).
$
\end{enumerate}
\end{lem}

\begin{proof}  (i) Use the identity $\sum_{d \mid n}\varphi(d)=n$ to eliminate the inner double sum:
\begin{equation} \label{44.095}
\sum_{d \mid p-1,} \sum_{\gcd (n,(p-1)/d)=1}1 =\sum_{d \mid p-1} \varphi((p-1)/d)=p-1.
\end{equation} 
Substituting this returns
\begin{equation}  \label{44.019}
\sum_{v\leq x} \frac{1}{\varphi(p^2)}\sum_{d \mid p-1,} \sum_{\gcd(n,p-1)=1} 1= \sum_{v\leq x} \frac{1}{p} 
= \frac{1}{p} (x -\{x\}).
\end{equation}  
Lastly, take the identity $[x]=x-\{x\}$, where $\{x\}$ is the fractional function, to complete the proof. (ii) The proof of this case is similar.
\end{proof}

%sssssssssssssssssssssssssssssssssssssssssssssssssssssssssssssssssssssssss
\subsection{Sums Over The Bases And Primes}

\begin{lem}  \label{lem44.3}  Let \(x\geq 1\) be a large number, and let \(\varphi (n)\) be the Euler totient function. Then

\begin{enumerate} [font=\normalfont, label=(\roman*)]
\item $ \displaystyle  \frac{1}{x} \sum_{v \leq x, } \sum_{p \leq x}\sum_{d \mid p-1,} \frac{1}{\varphi(p^2)}\sum_{\gcd(n,(p-1)/d)=1} 1= \log \log x +b_0 +O\left ( \frac{1}{\log x} \right ).$  
\item $ \displaystyle
\frac{1}{x} \sum_{v \leq x, } \sum_{p \leq x} \frac{1}{\varphi(p^2)}\sum_{\gcd(n,p-1)=1} 1=a_0 \log \log x +a_1 +O\left ( \frac{1}{\log x} \right ), $ \\
where $a_0>0, a_1$ and $b_0>0$ are constants.
\end{enumerate}
\end{lem}

\begin{proof} (i) Using the identity $\sum_{d \mid p-1}\varphi(d)=p-1$ is used to eliminate the inner double sum yield\\
\begin{eqnarray} \label{44.525}
\frac{1}{x}\sum_{v\leq x,} \sum_{p \leq x }\frac{1}{\varphi(p^2)}\sum_{d \mid p-1,} \sum_{\gcd (n,(p-1)/d)=1}1 &=&\frac{1}{x} \sum_{p \leq x, } \sum_{ v\leq x} \frac{1}{\varphi(p^2)}\sum_{d \mid p-1,} \sum_{\gcd (n,(p-1)/d)=1}1  \nonumber \\
&=&\frac{1}{x} \sum_{p \leq x, } \sum_{ v\leq x} \frac{p-1}{\varphi(p^2)}  \nonumber \\
&=&\frac{1}{x}\sum_{ p\leq x} \left ( \frac{x}{p }  +O\left ( \frac{1}{p}\right )\right ) .
\end{eqnarray} 
Applying Lemma \ref{lem22.1} yields
\begin{equation} \label{44.528}
\sum_{ p\leq x} \frac{1}{p } +O\left ( \frac{1}{ x}  \sum_{ p\leq x}1\right )=\log \log x +b_0 + O\left ( \frac{1}{\log x} \right ) ,
\end{equation} 
where $b_0>0$ is a constant. (ii) The proof of this case is similar, but it uses Lemma \ref{lem22.3}.
\end{proof}

\subsection{Sums Over The Integers}
\begin{lem}  \label{lem44.5}  Let \(x\geq 1\) be a large number, and let \(\varphi (n)\) be the Euler totient function. Then
\begin{enumerate} [font=\normalfont, label=(\roman*)]
\item  $ \displaystyle \sum_{n\leq x,}  \sum_{1 \leq i \leq \xi(n),} \frac{1}{\varphi(n^2)}  \sum_{d \mid \lambda(n),} \sum_{\gcd (n,\lambda(n)/d)=1}1=\log x +\gamma  +O \left (\frac{1}{x}\right )  ,$ \\
where $\gamma>0$ is Euler constant.
\item $ \displaystyle \sum_{n\leq x,}  \sum_{1 \leq i \leq \xi(n)} \frac{1}{\varphi(n^2)}   \sum_{\gcd (n,\lambda(n))=1}1\gg \frac{\log x}{\log \log x}.$
\end{enumerate}
\end{lem}

\begin{proof} (i) Use the identity $\sum_{d \mid n}\varphi(d)=n$ to eliminate the inner double sum in the following way: 
\begin{equation} \label{4295}
\sum_{d \mid \lambda(n),} \sum_{\gcd (n,\lambda(n)/d)=1}1 =\sum_{d \mid \lambda(n)} \varphi(\lambda(n)/d)=\lambda(n).
\end{equation} 
Substituting this, and using the identities $\varphi(n^2)=n\varphi(n)$, and $\varphi(n)=\xi(n)\lambda(n)$ return
\begin{equation} \label{4295}
\sum_{n\leq x} \sum_{1 \leq i \leq \xi(n)} \frac{1}{\varphi(n^2)}  \cdot \lambda(n) =\sum_{p\leq x} \frac{1}{\varphi(n^2)}  \cdot \xi(n) \lambda(n)= \sum_{ n\leq x}  \frac{1}{n}.
\end{equation} 
Lastly, apply the usual formula to the harmonic sum. (ii) The proof of this case is similar: 
\begin{eqnarray}
\sum_{n\leq x,}  \sum_{1 \leq i \leq \xi(n)} \frac{1}{\varphi(n^2)}   \sum_{\gcd (n,\lambda(n))=1}&=&\sum_{n\leq x} \sum_{1 \leq i \leq \xi(n)} \frac{1}{\varphi(n^2)} \cdot  \varphi(\lambda(n))  \nonumber \\
&=&\sum_{n\leq x} \frac{\xi(n)}{n\varphi(n)} \cdot  \varphi( \lambda(n))  \nonumber \\ 
&=&\sum_{n\leq x} \frac{1}{n} \cdot  \frac{\varphi( \lambda(n))}{ \lambda(n)}   \\
&=&\sum_{n\leq x} \frac{1}{n} \prod_{p \mid \lambda(n)} \left (1 - \frac{1}{ p} \right ) \nonumber \\
&\gg&\sum_{n\leq x} \frac{1}{n}  \frac{1}{\log \log n}  \nonumber \\
&\gg& \frac{\log x}{\log \log x}  ,\nonumber 
\end{eqnarray}
but it has no simple exact form.
\end{proof}

%ppppppppppppppppppppppppppppppppppppppppppppppppppppppppppppppppppppppp
\subsection{Problems}
\begin{enumerate}
\item Determine an exact asymptotic formula for $$\sum_{n\leq x,}  \sum_{1 \leq i \leq \xi(n)} \frac{1}{\varphi(n^2)}   \sum_{\gcd (n,\lambda(n))=1}1=c_0\frac{\log x}{\log \log x}\left (1 +O\left ( \frac{1}{(\log \log x)^2 }\right ) \right )  ,$$ where $c_0>0$ is a constant.
\end{enumerate}

%55-55-55-55-55-55-55-55-55-55-55-55-55-55-55-55-55-55-55-55-55-55-55-55-55-55-55-55-55
\newpage
\section{Estimates For The Error Terms} \label{s55}

Upper bounds for the error terms occurring in the proofs of several results as Theorem \ref{thm1.1} to Theorem \ref{thm1.4} are determined here. 

%sssssssssssssssssssssssssssssssssssssssssssssssss
\subsection{Error Terms In Long Intervals}
The estimates for the long interval $[1,x]$ computed here are weak, but sufficient in many applications.\\

\begin{lem} \label{lem55.1}  Let \(x\geq 1\) be large number. Let \(p\geq 2\) be a large prime, and let $\tau \in \left ( \mathbb{Z} / p^2 \mathbb{Z} \right )^{\times} $ be a primitive root mod \(p^2\). 
If the element \(v\geq 2 \) and $\gcd(v,\varphi(p^2))=w$, then, 
\begin{enumerate} [font=\normalfont, label=(\roman*)]
\item $ \displaystyle \sum_{ p\leq x,}\sum_{d \mid p-1,} \sum_{\gcd (n,(p-1)/d)=1} 
\frac{1}{\varphi(p^2)}\sum_{1\leq m< \varphi(p^2)} e^{\frac{i 2\pi  (\tau ^{pdn}-v)m}{\varphi(p^2)}}
 \leq  2v \log \log x,$
\item $ \displaystyle \sum_{ p\leq x,} \sum _{\gcd (n,p-1)=1} \frac{1}{\varphi(p^2)}\sum_{1\leq m< \varphi(p^2)} e^{\frac{i 2\pi  (\tau ^{pdn}-v)m}{\varphi(p^2)} }\leq  2v\log \log x,$ \\

where $w \leq v$, for all sufficiently large numbers $x\geq 1$.
\end{enumerate}
\end{lem}
\begin{proof}  (i) Rearrange the inner triple finite sum in the form
\begin{eqnarray} \label{eq55.78}
E(x)&=& \sum_{  p\leq x} \sum _{d \mid p-1,}\sum _{\gcd (n,(p-1)/d)=1} \frac{1}{\varphi(p^2)}\sum_{1\leq m< \varphi(p^2)} e^{\frac{i 2\pi  (\tau ^{dpn}-v)m}{\varphi(p^2)} }  \\  
&= & \sum_{ p\leq x}  \frac{1}{\varphi(p^2)} \sum_{ 0<m< \varphi(p^2)} e^{-i 2 \pi \frac{vm}{\varphi(p^2)}} \sum _{d \mid p-1,}\sum _{\gcd (n,(p-1)/d)=1} e^{i 2 \pi 
\frac{m\tau ^{dpn}}{\varphi(p^2)}}  \nonumber\\
&=& \sum_{ p\leq x}  \frac{1}{\varphi(p^2)} \sum_{ 0<m< \varphi(p^2)} e^{-i 2 \pi \frac{vm}{\varphi(p^2)}} \sum _{d \mid p-1}\left (\sum _{\gcd (n,(p-1)/d)=1} e^{i 2 \pi 
\frac{\tau ^{dpn}}{\varphi(p^2)}} +O\left ( p^{1/2} \log^3 p\right) \right ) \nonumber\\
&=& \sum_{ p\leq x}  \frac{1}{\varphi(p^2)} \left (\sum_{ 0<m< \varphi(p^2)} e^{-i 2 \pi \frac{vm}{\varphi(p^2)}} \sum _{d \mid p-1,}\sum _{\gcd (n,(p-1)/d)=1} e^{i 2 \pi 
\frac{\tau ^{dpn}}{\varphi(p^2)}} \right ) \nonumber\\
&& \qquad \qquad +O\left ( \sum_{ p\leq x}  \frac{1}{\varphi(p^2)} \sum_{ 0<m< \varphi(p^2)} e^{-i 2 \pi \frac{vm}{\varphi(p^2)}} \sum _{d \mid p-1} p^{1/2} \log^3 p \right ) 
\nonumber\\
&=&T_1+T_2.
\end{eqnarray} 
The third line in (\ref{eq55.78}) follows from Lemma \ref{lem33.22}. To complete the estimate, apply Lemma \ref{lem555.2} and Lemma \ref{lem555.3} to the terms $T_1$ and $T_2$ respectively. The proof of statement (ii) is similar.
\end{proof}

\begin{lem}  \label{lem555.2}  For any fixed integer $v \geq 2$, and a large number $x \geq 1$,
\begin{equation} \label{5090}
\left | \sum_{ p\leq x}  \frac{1}{\varphi(p^2)} \sum_{ 0<m< \varphi(p^2)} e^{-i 2 \pi \frac{vm}{\varphi(p^2)}} \sum _{d \mid p-1,}\sum _{\gcd (n,(p-1)/d)=1} e^{i 2 \pi 
\frac{\tau ^{dpn}}{\varphi(p^2)}} \right | \leq v \log \log x.
\end{equation} 
\end{lem}
\begin{proof} Trivially $\left | e^{i 2 \pi \tau ^{dpn}/\varphi(p^2)} \right | =1$, so the double inner inner sum reduces to
\begin{equation} \label{5095}
\sum _{d \mid p-1,} \sum _{\gcd (n,(p-1)/d)=1}1=p-1.
\end{equation} 
Plugging this trivial value returns
\begin{eqnarray} \label{eq555.18}
\left | T_1  \right | &\leq &  \left | \sum_{ p\leq x}  \frac{p-1}{\varphi(p^2)} \sum_{ 0<m< \varphi(p^2)} e^{-i 2 \pi \frac{vm}{\varphi(p^2)}} \right | \nonumber\\
&\leq & w\sum_{ p\leq x}  \frac{p-1}{\varphi(p^2)}\\
&=&w\sum_{ p\leq x}  \frac{1}{p} \nonumber \\
&=& w \log \log x \nonumber,
\end{eqnarray} 
where where $\varphi(p^2)=p(p-1)$, the parameter $\gcd(v,p(p-1))=w\leq v$, and 
\begin{equation} \label{5093}
\left |\sum_{ 0<m< \varphi(p^2)} e^{-i 2 \pi \frac{vm}{\varphi(p^2)}}\right | =w
\end{equation}
is an  exact evaluation.
\end{proof}

\begin{lem}  \label{lem555.3}  For any small number $\varepsilon >0$, a fixed integer $v \geq 2$, and a large number $x \geq 1$,
\begin{equation} \label{5097}
 \left | \sum_{ p\leq x}  \frac{1}{\varphi(p^2)} \sum_{ 0<m< \varphi(p^2)} e^{-i 2 \pi \frac{vm}{\varphi(p^2)}} \sum _{d \mid p-1} p^{1/2} \log^3 p \right | \leq    . 
\end{equation} 
\end{lem}
\begin{proof} The asymptotic estimate $ \sum _{d \mid p-1} =O(p^{\varepsilon})$ is the maximal number of divisors. Thus,
\begin{eqnarray} \label{eq555.18}
\left | T_2  \right | &\leq &  \left | \sum_{ p\leq x}  \frac{p^{1/2+\varepsilon} \log^3 p}{\varphi(p^2)} \sum_{ 0<m< \varphi(p^2)} e^{-i 2 \pi \frac{vm}{\varphi(p^2)}}   \right | \nonumber\\
&\leq & \sum_{ p\leq x}  \frac{p^{1/2+\varepsilon} \log^3 p}{\varphi(p^2)}\\
&\leq&\sum_{ p\leq x}  \frac{1}{p} \nonumber \\
&=&  \log \log x \nonumber,
\end{eqnarray} 
where $\varphi(p^2)=p(p-1)$, and $\sum_{ 0<m< \varphi(p^2)} e^{-i 2 \pi vm/\varphi(p^2)}=-1$.
\end{proof}

\begin{lem} \label{lem55.4}  Let \(x\geq 1\) be large number. Let \(p\geq 2\) be a large prime, and let $\tau \in \left ( \mathbb{Z} / p^k \mathbb{Z} \right )^{\times} $ be a primitive root mod \(p^2\). 
If the element \(v\geq 2 \) and $\gcd(v,\varphi(p^k))=w$, then,
\begin{enumerate} [font=\normalfont, label=(\roman*)]
\item $ \displaystyle
\sum_{ p\leq x,}\sum _{d \mid p-1,} \sum _{\gcd (n,(p-1)/d)=1} \frac{1}{\varphi(p^k)}\sum_{1\leq m< \varphi(p^k)} e^{\frac{i 2\pi  (\tau ^{p^{k-1}dn}-v)m}{\varphi(p^k)} }= O\left ( \log \log x \right),$
\item and
$ \displaystyle
\sum_{ p\leq x,} \sum _{\gcd (n,p-1)=1} \frac{1}{\varphi(p^k)}\sum_{1\leq m< \varphi(p^k)} e^{\frac{i 2\pi  (\tau ^{p^{k-1}dn}-v)m}{\varphi(p^k)} }= O\left ( \log \log x \right),$ \\

for all sufficiently large numbers $x\geq 1$.
\end{enumerate}
\end{lem}

\begin{proof} Generalize the proof of Lemma \ref{lem55.1} to fit the finite ring $\Z/p^k \Z$. \end{proof}

%ssssssssssssssssssssssssssssssssssssssssssssssssssssssssssssssssssssss
\subsection{Error Terms In Short Intervals}
This calculations show that the error terms for short intervals $[x,x+z]$, with $z=O(x)$ are nontrivials, and easy to determine using elementary method. But fail for very large intervals $[x,x^D]$, with $D>1$. 
\begin{lem} \label{lem55.3}  Let \(x\geq 1\) and $z \geq 1$ be large numbers. Let \(p\geq 2\) be a large prime, and let $\tau \in \left ( \mathbb{Z} / p^2 \mathbb{Z} \right )^{\times} $ be a primitive root mod \(p^2\). If the element \(v\geq 2 \) and $\gcd(v,\varphi(p^2))=w$, then,
\begin{enumerate} [font=\normalfont, label=(\roman*)]
\item $ \displaystyle
\sum_{ x \leq p\leq x+z,}\sum _{d \mid p-1,} \sum _{\gcd (n,(p-1)/d)=1} \frac{1}{\varphi(p^2)}\sum_{1\leq m< \varphi(p^2)} e^{\frac{i 2\pi  (\tau ^{dpn}-v)m}{\varphi(p^2)} }=O \left ( \frac{z^{1/2}}{x^{1/2} \log x}\right ) ,$
\item and 
$ \displaystyle
\sum_{ x \leq p\leq x+z,}\sum _{\gcd (n,p-1)=1} \frac{1}{\varphi(p^2)}\sum_{1\leq m< \varphi(p^2)} e^{\frac{i 2\pi  (\tau ^{dpn}-v)m}{\varphi(p^2)} }=O \left ( \frac{z^{1/2}}{x^{1/2} \log x}\right ) $ \\
for all sufficiently large numbers $x\geq 1.$
\end{enumerate}
\end{lem}
\begin{proof}   (i) Use the value $\varphi(p^2)=p(p-1)$ to rearrange the triple finite sum in the form
\begin{eqnarray} \label{55.999}
&&\sum_{ x \leq p\leq x+z,}\sum _{d \mid p-1,}\sum _{\gcd (n,(p-1)/d)=1} \frac{1}{\varphi(p^2)}\sum_{1\leq m< \varphi(p^2)} e^{\frac{i 2\pi  (\tau ^{dpn}-v)m}{\varphi(p^2)} } \\  
&= & \sum_{ x \leq p\leq x+z,} \left (\frac{1}{p}\sum_{ 0<m< \varphi(p^2)} e^{\frac{-i 2\pi vm}{\varphi(p^2)} } \right ) \left ( \frac{1}{p-1}\sum _{d \mid p-1,}\sum _{\gcd (n,(p-1)/d)=1} e^{\frac{i 2\pi  \tau ^{dpn}m}{\varphi(p^2)} } \right )\nonumber .
\end{eqnarray} 
Let
\begin{equation} \label{55.008}
A_p= \frac{1}{p}\sum_{ 0<m< \varphi(p^2)} e^{-i 2 \pi v m/\varphi(p^2)}  
\end{equation} 
and
\begin{equation} \label{55.009}
B_p=\frac{1}{p-1}\sum _{d \mid p-1,}\sum _{\gcd (n,(p-1)/d)=1} e^{i 2 \pi m\tau ^{dpn}/\varphi(p^2)}.
\end{equation} 
Utilize the prime number theorem $\pi(x)=x/\log x+O(x/\log^2 x) $ for $x \geq 1$, and the exact value of finite sum $\sum_{ 0<m< \varphi(p^2)} e^{-i 2 \pi vm/\varphi(p^2)} =-w$ whenever $\gcd(v,p(p-1))=w$, to estimate the first sum as  
\begin{eqnarray} \label{55.018}
\sum_{ x \leq p\leq x+z,} |A_p|^2 &=&\sum_{p \leq x} \left |  \frac{1}{p}\sum_{ 0<m< \varphi(p^2)} e^{-i 2 \pi v m/\varphi(p^2)} \right |^2 \nonumber\\ 
&=&  \sum_{ x \leq p\leq x+z,}  \frac{w}{p^2} \\
&=& w\int_x^{x+z}\frac{1}{t^2} d\pi(t) \nonumber \\
&=&O \left ( \frac{1}{x \log x} \right ) \nonumber,
\end{eqnarray}
where $w \leq v$ is a fixed number. Similarly, use the upper bound $\pi(x+z)-\pi(x) \leq 2z /\log x$ for $x \geq 1$, and $z \geq x^{3/4}$; and $\sum _{d \mid p-1} \varphi(p-1)=p-1$ for $p \geq 3$, to obtain a trivial estimate for the second sum as  
\begin{eqnarray} \label{55.019}
\sum_{ x \leq p\leq x+z,} |B_p|^2 &=&\sum_{ x \leq p\leq x+z,} \left |  \frac{1}{p-1}\sum _{d \mid p-1,}\sum _{\gcd (n,(p-1)/d)=1} e^{i 2 \pi m\tau ^{dpn}/\varphi(p^2)}\right |^2 \nonumber\\ 
&\leq&  \sum_{ x \leq p\leq x+z,}  \left | \frac{1 }{p-1}\sum _{d \mid p-1} \varphi((p-1)/d) \right |^2\\
&=& \sum_{ x \leq p\leq x+z,}  1 \nonumber\\ 
&\leq &  \frac{2z}{ \log x}\nonumber .
\end{eqnarray}
Now apply the Cauchy inequality 
\begin{eqnarray} \label{55.039}
\left | \sum_{ x \leq p\leq x+z,} A_p \cdot B_p \right |  &\leq & \left ( \sum_{ x \leq p\leq x+z,} |A_p|^2 \right )^{1/2} \cdot  \left ( \sum_{ x \leq p\leq x+z,} |B_p|^2 \right )^{1/2}\nonumber\\ 
&\ll&  \left ( \frac{1}{x \log x} \right )^{1/2}  \cdot  \left (  \frac{2z}{ \log x} \right )^{1/2} \\
&\ll&   \frac{z^{1/2}}{x^{1/2} \log x} \nonumber.
\end{eqnarray}
 (ii) The proof of this case is similar.
\end{proof}

%66-66-66-66-66-66-66-66-66-66-66-66-66-66-66-66-66-66-66-66-66-66-66-66-66-66
\newpage
\section{Counting Function For The Wieferich Primes} \label{s66}
The subset of primes $\mathcal{W}_2=\left\{ p:\ord_{p^2}(2) \mid p-1 \right \}=\{1093, 3511, \ldots ,\}$ associated with the base $v=2$ is the best known case. But, many other bases have been computed too, see \cite{DK11}, \cite{KR05}.

%bbbbbbbbbbbbbbbbbbbbbbbbbbbbbbbbbbbbbbbbbbbbbbbbbbbbbbbbbbbbbbbbbbb
\subsection{Proof Of Theorem \ref{thm1.1}}
\begin{proof} (Theorem \ref{thm1.1}) Let \(x\geq 1\) be a large number, and fix an integer \(v\geq 2\). Consider the sum of the characteristic function for the fixed element $v$ of order $\ord_{p^2}(v) \mid p-1$ over the primes in the short interval \([x,x+z]\). Then \\
\begin{equation} \label{66.003}
W_{v}(x+z)-W_{v}(x)= \sum_{x \leq p \leq x+z} \Psi_v (p^2).
\end{equation}\\
Replacing the characteristic function, see Lemma \ref{lem33.12}, and expanding the difference equation (\ref{66.003}) yield\\
\begin{eqnarray} \label{66.005}
\sum_{x \leq p \leq x+z} \Psi_v (p^2) 
&=&\sum_{x \leq p \leq x+z,}\sum_{d \mid p-1,} \sum _{\gcd (n,(p-1)/d)=1} \frac{1}{\varphi(p^2)}\sum_{0\leq m< \varphi(p^2)} e^{\frac{i 2\pi  (\tau ^{pn}-v)m}{\varphi(p^2)} }\nonumber \\
&=&\sum_{x \leq p \leq x+z}\frac{1}{\varphi(p^2)} \sum_{d \mid p-1,}\sum_{\gcd (n,(p-1)/d)=1}1   \\   
&& +    \sum_{x \leq p \leq x+z,}\sum_{d \mid p-1,}\sum _{\gcd (n,(p-1)/d)=1} \frac{1}{\varphi(p^2)}\sum_{1\leq m<\varphi(p^2)} e^{\frac{i 2\pi  (\tau ^{pn}-v)m}{\varphi(p^2)} } \nonumber\\
&=&M_v(x,z)\quad + \quad E_v(x,z) \nonumber.
\end{eqnarray} 
The main term $M_v(x,z)$ is determined by the index $m=0$, and the error term $E_v(x,z)$ is determined by the range $1 \leq m< \varphi(p^2)$. Applying Lemma \ref{lem44.1} to the main term and applying Lemma \ref{lem55.3} to the error term yield\\
\begin{eqnarray} \label{66.008}
& &  M_v(x,z) \quad +\quad  E_v(x,z) \\
&=& c_v \left ( \log \log (x+z)- \log \log (x) \right )    +O \left (\frac{1}{\log x}\right )   +O \left (\frac{z^{1/2}}{x^{1/2}\log x}\right )  \nonumber,
\end{eqnarray} 
Next, assuming that $z =O(x)$ it reduces to
\begin{equation} \label{66.008}
W_{v}(x+z)-W_{v}(x)= c_v\left (  \log \log (x+z)-\log \log (x) \right ) +O \left (\frac{1}{\log x}\right ),
\end{equation} 
where $c_v \geq 0$ is the density constant.
\end{proof}

The specific constant $c_v\geq 0$ for a given fixed base $v\geq 2$ is a problem in algebraic number theory, see Theorem \ref{thm77.1} for some details. 

%bbbbbbbbbbbbbbbbbbbbbbbbbbbbbbbbbbbbbbbbbbbbbbbbbbbbbbbbbbb
\subsection{Proof Of Theorem \ref{thm1.2}}
\begin{proof} (Theorem \ref{thm1.2}) Let \(x\geq 1\) be a large number, and fix an integer \(v\geq 2\). The sum of the characteristic function for the fixed element $v$ of order $\ord_{p^2}(v) \mid p-1$ over the primes in the interval \([1,x]\) is written as 
\begin{equation} \label{66.103}
W_{v}(x)= \sum_{ p \leq x} \Psi_0 (p^2).
\end{equation}
Replacing the characteristic function, see Lemma \ref{lem33.12}, and expanding yield
\begin{eqnarray} \label{66.105}
\sum_{ p \leq x} \Psi_v (p^2) 
&=&\sum_{ p \leq x,}\sum_{d \mid p-1,} \sum _{\gcd (n,(p-1)/d)=1} \frac{1}{\varphi(p^2)}\sum_{0\leq m< \varphi(p^2)} e^{\frac{i 2\pi  (\tau ^{pn}-v)m}{\varphi(p^2)} }\nonumber \\
&=&\sum_{ p \leq x}\frac{1}{\varphi(p^2)} \sum_{d \mid p-1,}\sum_{\gcd (n,(p-1)/d)=1}1   \\   
&& +    \sum_{x \leq p \leq x+z,}\sum_{d \mid p-1,}\sum _{\gcd (n,(p-1)/d)=1} \frac{1}{\varphi(p^2)}\sum_{1\leq m<\varphi(p^2)} e^{\frac{i 2\pi  (\tau ^{pn}-v)m}{\varphi(p^2)} } \nonumber \\
&=&M_v(x)\quad + \quad E_v(x) \nonumber.
\end{eqnarray} 
The main term $M_v(x)$ is determined by the index $m=0$, and the error term $E_v(x)$ is determined by the range $1 \leq m< \varphi(p^2)$. Applying Lemma \ref{lem44.1} to the main term and applying Lemma \ref{lem55.1} to the error term yield\\
\begin{eqnarray} \label{66.008}
W_{v}(x) &= &  M_v(x) \quad +\quad  E_v(x) \nonumber\\
&\leq & 2\log \log (x) +2v\log \log x  \nonumber\\
&\leq &4v\log \log x  \nonumber.
\end{eqnarray} 
This verifies the upper bound.
\end{proof}
%bbbbbbbbbbbbbbbbbbbbbbbbbbbbbbbbbbbbbbbbbbbbbbbbbbbbbbbbbbbbbbbbbbbbbbb
\subsection{Wieferich Constants}
An appropiate upper bound of a primes counting problem immediately provides information on the convergence of the infinite series $\sum_{p \geq 2} 1/p$. The best known case is the Brun constant 
\begin{equation}
B=\sum_{\text{twin primes } p , p+2} \frac{1}{p}=\frac{1}{3}+\frac{1}{5}+\frac{1}{7}+\cdots \approx 1.902160583104 \ldots, 
\end{equation}
see \cite[p.\ 9]{SG02}. A conjecture claims that $1.90216054<B<1.90216063$, see \cite[p.\ 215]{KD07}. However, the arithmetic properties, such as rationality and irrationality, of $B$ remains hopelessly unresolved. Similar problems arise fnrthe sequences of Wieferiech primes.\\

\begin{cor} \label{cor66.6} Let $v\geq 2$ be a small fixed integer. Then, \begin{equation} \sum_{\substack{p \geq 2 \\ v^{p-1}-1\equiv 0 \bmod p^2}} \frac{1}{p} < \infty.\end{equation} In particular, 
\begin{enumerate} [font=\normalfont, label=(\roman*)]
\item $ \displaystyle 
\sum_{\substack{p \geq 2 \\ 2^{p-1}-1\equiv 0 \bmod p^2}} \frac{1}{p}=\frac{1}{1093}+\frac{1}{3511}+ \frac{8c \log \log (10^{15})}{10^{15}} <0.00119974 $, \\
for some small constant $c>0$.
\item $\displaystyle 
\sum_{\substack{p \geq 2 \\ 3^{p-1}-1\equiv 0 \bmod p^2}} \frac{1}{p}=\frac{1}{11}+\frac{1}{1006003}+ \frac{12c \log \log (10^{14})}{10^{14}}< 0.0909102\nonumber, $\\
for some small constant $c>0$.
\end{enumerate}
\end{cor}
\begin{proof} (i) For base $v=2$, and using the numerical data in \cite{DK11}, the series is written a sum of two simpler subsums, and each subsums is evaluated or estimated:
\begin{eqnarray}
\sum_{\substack{p \geq 2 \\ 2^{p-1}-1\equiv 0 \bmod p^2}} \frac{1}{p}&=&\sum_{\substack{p \leq 10^{15} \\ 2^{p-1}-1\equiv 0 \bmod p^2}} \frac{1}{p} +\sum_{\substack{p > 10^{15} \\ 2^{p-1}-1\equiv 0 \bmod p^2}} \frac{1}{p}  \nonumber\\
&=& \frac{1}{1093}+\frac{1}{3511}+\sum_{\substack{p > 10^{15} \\ 2^{p-1}-1\equiv 0 \bmod p^2}} \frac{1}{p}  \\
&=& \frac{1}{1093}+\frac{1}{3511}+ \int_{10^{15}}^{\infty}\frac{1}{t}d W_2(t) \nonumber,
\end{eqnarray}
where $W_2(t) \leq 8 \log \log t$. 
\end{proof} 

%bbbbbbbbbbbbbbbbbbbbbbbbbbbbbbbbbbbbbbbbbbbbbbbbbbbbbbbbbbbbbbbbbbbbbbbbbbbbbbbb
\subsection{Proof Of Theorem  \ref{thm1.3}}
Some numerical data has been compiled for this case in \cite{MP93} and \cite{KR05}. For examples, $42^{22}-1 \equiv 0 \bmod 23^3$, and $68^{112}-1 \equiv 0 \bmod 113^3$, for the ranges $v<100$, and $p<2^{32}$. But, these are very rare. 
\begin{proof} (Theorem \ref{thm1.3}) Let \(x\geq 1\) be a large number, and fix a small integer \(v\geq 2\).  Let $k=3$ in Lemmas \ref{lem33.7}, and 

\ref{lem33.12}, to obtain a characteristic function to fit the finite ring $\Z/p^3\Z$. Summing the characteristic function for the fixed element $v$ of 

order $\ord_{p^2}(v) \mid p-1$ over the primes in the interval \([1,x]\) yields
\begin{equation} \label{66.303}
\sum_{p \leq x} \Psi_0 (p^3)=\sum_{p \leq x,}\sum_{d \mid p-1,} \sum _{\gcd (n,(p-1)/d)=1} \frac{1}{\varphi(p^3)}\sum_{0\leq m< \varphi(p^3)} e^{\frac{i 2\pi  (\tau ^{dp^2n}-v)m}{\varphi(p^3)} }.
\end{equation}
Break the quadruple sum into a main term and an error term. The main term $M_v(x)$ is determined by the index $m=0$, and the error term is determined by the range $1 \leq m< \varphi(p^3)$, that is
\begin{eqnarray} \label{66.305}
&&\sum_{p \leq x,}\sum_{d \mid p-1,} \sum _{\gcd (n,(p-1)/d)=1} \frac{1}{\varphi(p^3)}\sum_{0\leq m< \varphi(p^3)} e^{\frac{i 2\pi  (\tau ^{dp^2n}-v)m}{\varphi(p^3)} }\nonumber \\
&=&\sum_{ p \leq x}\frac{1}{\varphi(p^3)} \sum_{d \mid p-1,}\sum_{\gcd (n,(p-1)/d)=1}1  \\   
&& +    \sum_{ p \leq x} \frac{1}{\varphi(p^3)}\sum_{d \mid p-1,}\sum _{\gcd (n,(p-1)/d)=1}\sum_{1\leq m<\varphi(p^3)} e^{\frac{i 2\pi  (\tau ^{dp^2n}-v)m}{\varphi(p^3)} }  \nonumber.
\end{eqnarray} 
Use the identity $\sum_{d \mid n} \varphi(d)=n$ and $\varphi(p^2)=p(p-1)$ to verify that the main term is finite:
\begin{eqnarray} \label{66.307}
\sum_{ p \leq x}\frac{1}{\varphi(p^3)} \sum_{d \mid p-1,}\sum_{\gcd (n,(p-1)/d)=1}1 
&=&\sum_{ p \leq x}\frac{1}{\varphi(p^3)} \cdot (p-1) \nonumber\\
&=&  \sum_{ p \leq x}\frac{1}{p^2 } \\
&=&O(1) \nonumber .
\end{eqnarray} 
Similarly, to verify that the error term is finite,  use the value $\sum_{1\leq m< \varphi(p^3)} e^{i 2\pi  vm/\varphi(p^3) } =-w$ whenever $\gcd(v,p^2(p-1))=w$ , and the trivial estimate for the double inner sum:
\begin{eqnarray} \label{66.309}
& & \sum_{ p \leq x}\frac{1}{\varphi(p^3)}\left |\sum_{d \mid p-1,}\sum _{\gcd (n,(p-1)/d)=1,} \sum_{1\leq m<\varphi(p^3)} e^{\frac{i 2\pi  (\tau ^{dp^2n}-v)m}{\varphi(p^3)} }  \right | \nonumber \\
&\leq & \sum_{ p \leq x}\frac{1}{\varphi(p^3)}\left |\sum_{1\leq m<\varphi(p^3)} e^{\frac{-i 2\pi  vm}{\varphi(p^3)} }  \right |\left | \sum_{d \mid p-1,}\sum _{\gcd (n,(p-1)/d)=1}  e^{\frac{i 2\pi  \tau ^{dp^2n}m}{\varphi(p^3)} } \right | \nonumber \\
&\leq&\sum_{ p \leq x}\frac{1}{\varphi(p^3)} \cdot w \cdot  (p-1) \nonumber\\
&\leq&2w  \sum_{ p \leq x}\frac{1}{p^2 } \\
&=&O(1) \nonumber ,
\end{eqnarray} 
where $w \leq v$ is a small fixed number. Therefore
\begin{equation} \label{66.313}
\sum_{p \leq x} \Psi_0 (p^3)=O(1),
\end{equation}\\
and this implies that the number of primes such that $v^{p-1}-1 \equiv 0\bmod p^3$ is finite.
\end{proof}

%ppppppppppppppppppppppppppppppppppppppppppppppppppppppppppppppppppppppp
\subsection{Problems}
\begin{enumerate}
\item Modify the characterisitc function in Lemma \ref{lem33.7}, suitable for $\Z/p^{n+3}\Z$, to prove that the subset of primes $\mathcal{A}(v)= \{p : v^{p^n(p-1)}-1 \equiv 0 \bmod p^{n+3}\}$ is finite, $n \geq 0$. \\
\item A Wieferich prime pairs $p$ and $q$ satisfies the reciprocity condition $$p^{q-1} -1 \equiv 0 \bmod q^2 \quad  \text{ and } \quad q^{p-1} -1 \equiv 0 \bmod q^2.$$
Many pairs of these primes are known, for example, $$(p,q)=(83,4871; (2903,18787); (911,318917) \leq 10^6$$ are known, see \cite{MP03}, and \cite[p.\ 935]{KR05}. Prove that there are infinitely many, and give an estimate of its counting function.\\
\item Let $v\geq 2$ be a small fixed integer. Apply the $abc$ conjecture to show that the subset of primes $\{ p : v^{p-1}-1 \equiv 0 \bmod p^3\}$ is finite.
\item Let $p>3$ be a prime. Use the Wilson result $(p-1)! \equiv -1 \bmod p$ to prove the Wolstenhome lemma, \cite[p.\ 94]{RD96}: $$1+\frac{1}{2}+\frac{1}{2}+\cdots +\frac{1}{p-1} \equiv 0 \bmod p^2.$$
\item Let $n>9$ be an integer. Use Gauss generalization $K(n)=\prod_{\gcd(k,n)=1}k \equiv \pm 1 \bmod n$ of the Wilson result $(p-1)! \equiv -1 \bmod p$ to prove: $$1+\frac{1}{a_1}+\frac{1}{a_2}+\cdots +\frac{1}{a_{\varphi(n)}} \equiv 0 \bmod n^2,$$ where $\gcd(a_k,n)=1$, if and only if $K(n)=-1$.
\item Let $B_k \in \Q$ be the $k$-th Bernoulli number, and let $p\geq 3$ be a prime. Generalize the power sum congruence $$\sum_{1 \leq m \leq p} m^k \equiv p B_k \bmod p,$$ see \cite{LE38}, to composite integers $p=n \geq 1$.
\end{enumerate}

%77-77-77-77-77-77-77-77-77-77-77-77-77-77-77-77-77-77-77-77-77-77-77-77-77-77-77-77-77-77-77-77-
\newpage
\section{Correction Factors}  \label{s77}
About a quarter century ago the Lehmers discovered a discrepancy in the theory of primitive roots. This discovery led to the concept of correction factors, see \cite{SP03} for the historical and technical details. This topic has evolved into a full fledged area of algebraic number theory, see \cite{LS14}, \cite[Chapter 2]{PW14}, \cite{BD11}, \cite{BJ17}, etc.  \\

The correction factors $c_v \geq 0$ accounts for the dependencies among the primes in Conjecture \ref{conj1.1} and Theorem \ref{thm1.1}. According to the analysis in \cite[Section 2]{KN15}, for any random integer $v \geq 2$, the corresponding correction factor associated to a random subset $\mathcal{W}_v$ of Wieferich primes has the value $c_v=1$ with 
probability one. Thus, the nonunity correction factors $c_v \ne1$ occur on a subset of integers $v \geq 2$ zero density. For example, at the odd prime powers $v=q^k \equiv 1 \bmod 4$. A more precise expression for the value of the correction factors is given below.

\begin{thm} \label{thm77.1} Let $v=ab^k$ with $a \geq 2$ squarefree. Then 
\begin{equation} \label{666}
c_v= \sum_{n \geq 1}\sum_{d \mid n} \frac{\mu(n)\gcd(dn,k)}{ dn \varphi(dn) },
\end{equation}
where $\mathcal{K}_{r}=\mathbb{Q}\left ( \zeta_{r}, v^{1/{r}} \right )$ is a number field extension, $\zeta_{r}$ is a primitive $r$th root of unity. 
\end{thm}

\begin{proof} Let $r=nd$ with $d \mid n-1$, and let $[\mathcal{K}_{r}:\mathbb{Q}]$ be the index of the finite extension. The splitting field of the pure 

equation $x^{p(p-1)/d}-v=0$ is the $r$th-cyclotomic numbers field $\mathcal{K}_{r}$. By the Frobenius density theorem, \cite[p.\ 134]{JG73}, (or the Chebotarev density theorem), the proportion of primes that split completely is $1/[\mathcal{K}_{nd}:\mathbb{Q}]$. Therefore, the proportion of primes that do not split completely is 
\begin{equation} \label{667}
1- \frac{1}{ [\mathcal{K}_{nd}:\mathbb{Q}] }.
\end{equation}
In light of this information, the inclusion-exclusion principle leads to
\begin{equation} \label{666}
c_v= \sum_{n \geq 1}\sum_{d \mid n} \frac{\mu(n)}{ [\mathcal{K}_{nd}:\mathbb{Q}] }.
\end{equation}
Furthermore, for $v=ab^k$ with $a \geq 2$ squarefree, the index has the following form
\begin{equation} \label{666}
[\mathcal{K}_{r}:\mathbb{Q}]=\left \{
\begin{array}{ll}
\frac{r \varphi(r)}{2 \gcd(k,r)} & \text{ if } r \mid 2a,  \text{ and } a \equiv 1 \bmod 4,\\
\frac{r \varphi(r)}{ \gcd(k,r)} & \text{ otherwise}, \\
\end{array} \right .
\end{equation}
see \cite[p.\ 214]{HC67}, \cite[p.\ 5]{AC14}.
\end{proof}

This technique explicates the fluctuations in the numbers of primes with respect to the bases $v \geq 2$, see \cite[Section 4.1]{DK11} for a discussion. The best known case is $v=5$. In this case, the dependency between two primes $2$ and $5$ is captured in the index calculation for $ r=10$. Here, the number field extension index fails to be multiplicative: 
\begin{equation} \label{666}
 20=[\mathcal{K}_{10}:\mathbb{Q}] \ne [\mathcal{K}_{2}:\mathbb{Q}] \cdot  [\mathbb{Q}(\zeta_{5}):\mathbb{Q}]=10
\end{equation}

In this case $\pm \sqrt{5}=\zeta_5+ \zeta_5^{-1}-\zeta_5^2-\zeta_5^{-2}$ so $\pm \sqrt{5}\in \mathbb{Q}(\zeta_{5})$, see \cite[p.\ 13]{PW14}, \cite[Chapter 2]{AC14}, \cite[Section 3]{SP03}, 

%bbbbbbbbbbbbbbbbbbbbbbbbbbbbbbbbbbbbbbbbbbbbbbbbbbbbbbbbbbb
\subsection{Problems}
\begin{enumerate}
\item Show that the ring of integers $\mathcal{O}_{K}$ of the numbers field $\Q(\sqrt[n]{2})$ is not $\Z[\sqrt[n]{2}]$ for $n=1093$ and $3511$. Hint: show that it contains half integers $(1+\sqrt[n]{2})/2$.     
\end{enumerate}

%88-88-88-88-88-88-88-88-88-88-88-88-88-88-88-88-88-88-88-88-88-88-88-88-88-88-88-88-88-88-88-88-88-88-88-88-88-88-88-88-8
\newpage
\section{Average Order Of A Random Subset Of Wieferich Primes}  \label{s9}
Quite often, the calculation of average density of an intractable primes distribution problems is the first line approach to solving the individual primes distribution problem, see \cite{SP69}, \cite{BD11}, \cite{KS13}, et cetera.\\

\begin{proof} (Theorem \ref{thm1.4}) Let \(x\geq 1\) be a large number, and let \(v\geq 2\) be a random integer. The average number of Wieferich primes in base $v$ over the interval \([x,x+z]\) is given by 
\begin{equation} \label{69003}
\frac{1}{x} \sum_{v \leq x } \left (W_{v}(x+z)-W_{v}(x) \right ) =\frac{1}{x} \sum_{v \leq x, } \sum_{x \leq  p\leq x+z} \Psi_0(v).
\end{equation}
Replacing the characteristic function, Lemma \ref{lem33.12}, and expanding the difference equation (\ref{69003}) yield\\
\begin{eqnarray} \label{69005}
\frac{1}{x} \sum_{v \leq x, } \sum_{x \leq  p\leq x+z} \Psi_0(v)&=&\frac{1}{x} \sum_{v \leq x, }  \sum_{x \leq  p\leq x+z}\sum _{d \mid p-1,}\sum _{\gcd (n,(p-1)/d)=1} \frac{1}{\varphi(p^2)}\sum_{0\leq m < \varphi(p^2)} e^{\frac{i 2\pi  (\tau ^{pn}-v)m}{\varphi(p^2)} } \nonumber\\
&=&\frac{1}{x} \sum_{v \leq x }  \sum_{x \leq  p\leq x+z} \frac{1}{\varphi(p^2)}\sum _{d \mid p-1,} \sum_{\gcd (n,(p-1)/d)=1}1    \\
&& +   \frac{1}{x} \sum_{v \leq x, } \sum_{x \leq  p\leq x+z}\sum _{d \mid p-1,} \sum _{\gcd (n,(p-1)/d)=1} \frac{1}{\varphi(p^2)}\sum_{1\leq m < \varphi(p^2)} e^{\frac{i 2\pi  (\tau ^{pn}-v)m}{\varphi(p^2)} } \nonumber \\
&=&M_0(x,z)\quad + \quad E_0(x,z) \nonumber.
\end{eqnarray} 
The main term $M_0(x,z)$ is determined by the index $m=0$, and the error term $E_0(x,z)$ is determined by the range $1 \leq m< \varphi(p^2)$. Applying Lemma \ref{lem44.1} to the main term and applying Lemma \ref{lem55.3} to the error term yield
\begin{eqnarray} \label{69008}
\frac{1}{x} \sum_{v \leq x } \left (W_{v}(x+z)-W_{v}(x) \right )  &= & M_0(x,z) \quad +\quad  E_0(x,z) \\
&=& \log  \log(x+z)-\log \log(x)  +O \left (\frac{1}{\log x} \right )    \nonumber\\
&&    \qquad \qquad+O \left (\frac{z^{1/2}}{x^{1/2}\log x}\right ) \nonumber.
\end{eqnarray} 
Next, assuming that $z =O(x)$ it reduces to
\begin{equation} \label{69008}
\frac{1}{x} \sum_{v \leq x } \left (W_{v}(x+z)-W_{v}(x)\right )  =\log  \log(x+z)-\log \log(x)  +O \left (\frac{1}{\log x}\right ) .
\end{equation} 

\end{proof}

%99999999999999999999999999999999999999999999999999999999999999
\newpage
\section{Balanced Subsets} \label{s99}
The balanced index $\ind_{p^2}(v) =p$ occurs whenever the base $v \geq 2$ has \textit{balanced} order $\ord_{p^2}(v) = p-1$.\\

The balanced subset is
\begin{equation} \label{99.690}
	\mathcal{B}_v= \left\{ p\leq x:\ord_{p^2}(v)=p-1 \right \} \nonumber.
\end{equation}
 For a large number $x \geq 1$, the corresponding counting function for the number of Wieferich primes up to $x$ with respect to a base of balanced order is defined by
\begin{equation} \label{99.692}
	B_{v}(x)=\#\left\{ p\leq x:\ord_{p^2}(v)=p-1 \right \} \nonumber.
\end{equation}

\begin{thm} \label{thm99.1} Let $v\geq 2$ be a base, and let  $x \geq 1$ and $z \geq x$ be large numbers. Then, the number of Wieferich primes $p$ such that $\ord_{p^2}(v)=p-1$  in the short inteval $[x, x+z]$ has the asymptotic formula
\begin{equation} \label{6022}
	B_v(x+z)-	B_v(x)=c_v \left ( \log \log (x+z)-\log \log( x) \right )+E_v( x,z),
\end{equation}
where $c_v \geq 0$ is the correction factor, and $E_v(x)$ is an error term.
\end{thm} 

\begin{proof} Let \(x\geq 1\) be a large number, and fix an integer \(v\geq 2\). Consider the sum of the characteristic function for the fixed element $v$ of order $\ord_{p^2}(v)= p-1$ over the primes in the short interval \([x,x+z]\). Then \\
\begin{equation} \label{99.663}
B_{v}(x+z)-B_{v}(x)= \sum_{x \leq p \leq x+z} \Psi (v).
\end{equation}\\
Replacing the characteristic function, Lemma \ref{lem33.5}, and expanding the existence equation (\ref{99.663}) yield\\
\begin{eqnarray} \label{99.605}
\sum_{x \leq p \leq x+z} \Psi (v) 
&=&\sum_{x \leq p \leq x+z,}\sum _{\gcd (n,p-1)=1} \frac{1}{\varphi(p^2)}\sum_{0\leq m< \varphi(p^2)} e^{\frac{i 2\pi  (\tau ^{pn}-v)m}{\varphi(p^2)} }\\
&=&\sum_{x \leq p \leq x+z}\frac{1}{\varphi(p^2)} \sum_{\gcd (n,p-1)=1}1   \nonumber\\   
&& \qquad \qquad+    \sum_{x \leq p \leq x+z,}\sum _{\gcd (n,p-1)=1} \frac{1}{\varphi(p^2)}\sum_{1\leq m<\varphi(p^2)} e^{\frac{i 2\pi  (\tau ^{pn}-v)m}{\varphi(p^2)} } \nonumber \\
&=&M_v(x,z)\quad + \quad E_v(x,z) \nonumber.
\end{eqnarray} 
The main term $M_v(x,z)$ is determined by the index $m=0$, and the error term $E_v(x,z)$ is determined by the range $1 \leq m< \varphi(p^2)$. Applying Lemma \ref{lem44.2} to the main term and applying Lemma \ref{lem55.1} to the error term yield
\begin{eqnarray} \label{99.608}
\sum_{x \leq p \leq x+z} \Psi(v)&= & M_v(x,z) \quad +\quad  E_v(x,z) \nonumber\\
&=& c_v \left ( \log \log (x+z)- \log \log x \right )    +O \left (\frac{1}{\log x}\right )   +O \left (\frac{z}{x\log x}\right )  \nonumber,
\end{eqnarray} 
Next, assuming that $z =O(x)$ it reduces to
\begin{equation} \label{99.668}
B_{v}(x+z)-B_{v}(x) = c_v \left ( \log \log (x+z)- \log \log x \right ) +O \left (\frac{1}{\log x}\right ),
\end{equation} 
where $c_v \geq 0$ is the density constant. 
\end{proof}
The average constant is $a_0=.37399581 \ldots ,$ see (\ref{3500}). The specific constant $c_v\geq 0$ for a given fixed base $v\geq 2$ is a problem in algebraic number theory explicated in the next section.

%10-10-10-10-10-10-10-10-10-10-10-10-10-10-10-1010-10-10-10-10-10-10-10-10-10-10-10-10-10-10-10-10-10-10-10-10-10-10-10
\newpage
\section{Data For Next Primes}  \label{s10}
The numerical data demonstrates that the number of primes $p \leq x=4 \times 10^{15}$ with respect to a fixed base $v \geq 2$ is a small quantity and vary from 0 to about 6. Some estimates for the intervals that contain the next Wieferich primes with respect to several fixed bases are sketched in this section.

\subsection{Calculations For The Subset $\mathcal{W}_2$}
It have taken about a century to determine the number of base $v=2$ primes up to $x=10^{15}$. The numerical data demonstrate that the subset is just

\begin{equation}
\mathcal{W}_2=\left\{ p:\ord_{p^2}(2) \mid p-1 \right \}=\{1093, 3511, \ldots ,\}
\end{equation}

see \cite{DK11}. To estimate the size of an interval $[10^{15},10^D]$ with $D>15$, that contains the next Wieferich prime, assume that the correction factor $c_2>0$ of the set of Wieferich primes $\mathcal{W}_2$ is the same as the average density $c_0=1$ of a random set of Wieferich primes
\begin{equation} \label{10.963}
\mathcal{W}_v=\left\{ p:\ord_{p^2}(v)=p-1 \right \}.
\end{equation}

Specifically, $c_2=c_0=1$. Let $x=10^{15}$ and let $x+z=10^{15w}$. By assumption, and Theorem \ref{thm1.2}, it follows that 
\begin{eqnarray}
1& \leq & c_2 \left ( \log \log (x+z) -\log \log x\right ) +E(x,z) \nonumber \\
&=&c_0 \left ( \log \log (x+z) -\log \log x\right ) +E(x,z)  \\
&=& 
 \log \frac{\log 10^{15w}}{\log 10^{15}} +E(x,z)\nonumber,
\end{eqnarray}

where $E(x,z)$ is an error term. Therefore, it is expected that the next Wieferich prime $p > 10^{15}$ is in the interval $[x=10^{15}, x+z=10^{40}]$.\\

\subsection{Calculations For The Subset $\mathcal{W}_5$}
The numerical data for the number of primes up to $x=10^{15}$ with respect to base $v=5$ demonstrates that the subset 
\begin{equation}
\mathcal{W}_5=\{2, 20771, 40487, 53471161, 1645333507, 6692367337, 188748146801, \ldots \},
\end{equation}
where each prime satisfies $5^{p-1}-1 \equiv 0 \bmod p^2$, see \cite{DK11}, contains about twice as many primes as the numerical data for many other subsets $\mathcal{W}_2, \mathcal{W}_3, \mathcal{W}_4, \mathcal{W}_6, \mathcal{W}_7,$ et cetera. A larger than average correction factor $c_5>1$, and many other bases $v$ explicates this disparity, see Theorem \ref{thm77.1}. \\

To estimate the size of an interval $[10^{15},10^D]$ with $D>15$, that contains the next base $v=5$ Wieferich prime, assume that the density constant $c_5>1$ of the set of Wieferich primes 
\begin{equation} \label{10.063}
\mathcal{W}_5=\left\{ p:\ord_{p^2}(5) \mid p-1 \right \}
\end{equation}
is the same as the average density $c_0=1$ of a random set of Wieferich primes
\begin{equation} \label{10.763}
\mathcal{W}_v=\left\{ p:\ord_{p^2}(v)=p-1 \right \}.
\end{equation}

Specifically, $c_5>c_0=1$. Let $x=10^{15}$ and let $x+z=10^{15w}$. By assumption, and Theorem \ref{thm1.2}, it follows that 
\begin{eqnarray}
1& \leq & c_5 \left ( \log \log (x+z) -\log \log (x)\right ) +O \left (\frac{1}{\log x}\right ) \nonumber \\
&=&c_5\left ( \log \log (x+z) -\log \log (x)\right ) +O \left (\frac{1}{\log x}\right )  \\
&=&c_5 
 \log \frac{\log 10^{15w}}{\log 10^{15}} +O \left (\frac{1}{\log x}\right )\nonumber.
\end{eqnarray}

Therefore, it is expected that the next Wieferich prime $p > 10^{15}$ is in the interval $[x=10^{15}, x+z=10^{40}]$.\\

\subsection{Calculations For The Balanced Subset}
The previous estimate assume that the order of the base $\ord_{p^2}(2) \mid p-1$ can vary as the prime $p$ varies. In contrast, if the order of the base remains exactly $\ord_{p^2}(2) = p-1$ as the prime $p$ varies, then the estimated interval is significantly larger as demonstrated below. The subset of balanced primes in base $v=2$ is

\begin{equation} \label{10.264}
\mathcal{B}_2=\left\{ p:\ord_{p^2}(2) = p-1 \right \}
\end{equation}
and its density is the same as the average density $a_0=.37399581 \ldots$ of a random set of Wieferich primes
\begin{equation} \label{10.363}
\mathcal{W}_v=\left\{ p:\ord_{p^2}(v) \mid p-1 \right \}.
\end{equation}
Specifically, $c_2=a_0$. Let $x=10^{15}$ and let $x+z=x^{15w}$. By assumption, and Theorem \ref{thm1.2}, it follows that 
\begin{eqnarray}
1& \leq & c_2 \left ( \log \log (x+z) -\log \log x\right ) +O \left (\frac{1}{\log x}\right ) \nonumber \\
&=&a_0 \left ( \log \log (x+z) -\log \log x\right ) +O \left (\frac{1}{\log x}\right )  \\
&=& 
.37399581 \log \frac{\log 10^{15w}}{\log 10^{15}} +O \left (\frac{1}{\log x}\right )\nonumber.
\end{eqnarray}

Therefore, it is expected that the next Wieferich prime $p > 10^{15}$, for which the order $\ord_{p^2}(2)= p-1$ and its index is $\ind_{p^2}(2) =p$, is in the interval $[x=10^{15}, x+z=10^{218}]$.

%12-12-12-12-12-12-12-12-12-12-12-12-12-12-12-12-12-12-12-12-12-12-12-12-12-12-12-12-12-12-12-12-12-12-12-
\newpage
\section{Least Primitive Root In Finite Rings} \label{s12}
The primitive roots $v\geq 2$ in residues finite rings modulo $p$ rarely fail to be a primitive roots in residues finite rings modulo $p^2$. 

\begin{dfn} {\normalfont
A primitive root $v=v(p)$ modulo $p$ is called \textit{nilpotent} if the congruence $v^{p-1} -1\equiv 0 \bmod p^2 $ holds. The subset of primes such that $v$ is nilpotent is denoted by $\mathcal{N}_v=\{p : \ord_{p}(v)=p-1 \text{ and } \ord_{p^2}(v)=p-1\}$. }
\end{dfn}

\begin{dfn} {\normalfont 
The least primitive root modulo a prime $p\geq 3$ is denoted by $g=g(p) \geq 2$ and least primitive root modulo a prime $p^2$ is denoted by $h=h(p)\geq 2$.}
\end{dfn}
The occurrence of nilpotent primitive roots $v=v(p)$ are very common. A sample is provided in the table.
\begin{center}
\begin{tabular}{ c c c c } 
 \hline
$v(p)$ &  $p$     & $v(p)$ &  $p$ \\
 \hline
3 &  1006003 &7&5\\ 
 5 &  40487 &10&487\\ 
 6 &  66161& 11&71 \\ 
 \hline
\end{tabular}
\end{center}
However, the occurrence fo nilpotent and least primitive roots simultaneously, are rarer. A complete list of the known cases is provided in the table. 
\begin{center}
\begin{tabular}{ c c c c } 
 \hline
Nilpotent $v(p)$ & Least  $g(p)$     & Least $h(p^2)$ &  Modulo $p$ or $p^2$\\
 \hline
5 &  5 &7&40487\\ 
 5 & 5 &7&6692367337\\ 
 \hline
\end{tabular}
\end{center}
The computational data for all the primes $p \leq 10^{15}$ are compiled in \cite[p.\ 929]{KR05} and \cite{PA09}, \cite{DK11}, et alii.
\begin{lem} \label{lem12.1} If $g \geq 2$ is a primitive root modulo a prime $p \geq 3$. Then, either $g$ or its inverse $\overline{g}$ is a primitive 

root modulo $p^k$ for every $k \geq 1$.
\end{lem}

The existence of nilpotent primitive roots, which causes the sporadic divisibility of the integers $g^{p-1}-1$ by prime powers $p^2$, has an interesting effect on the distribution of primitive roots in the cyclic groups $\left (\Z/p^k \Z \right )^{\times}$ with $k \geq 1$, and the noncyclic groups $\left (\Z/n \Z \right )^{\times}$ with $n \geq 1$ is an arbitrary integer. The precise criterion for cyclic groups is specified below.
\begin{lem} \label{lem12.3} A primitive root $g\geq 2$ modulo $p \geq 3$ is primitive root $g\geq 2$ modulo $p^k $ for all $k \geq 1$ if and only if $g^{p-1}-1 \not \equiv 0 \bmod p^2$.
\end{lem}

\begin{thm} \label{thm12.1} Let \(x \geq 1\) be a large number. Then
\begin{enumerate} [font=\normalfont, label=(\roman*)]
\item The number of primes such that $g(p)\geq 2$ is a primitive root modulo $p$, but $g(p^2)\geq g(p)+1$ is infinite.
\item The counting function has the asymptotic formula
\begin{equation} \label{69763}
\#\{p \leq x: g(p) \text{ and } g(p^2)\geq g(p)+1 \} \sim a_0 \log \log x,
\end{equation}
where $a_0=.37399581 \ldots$.
\end{enumerate} 
\end{thm}
\begin{proof} Without loss in generality, let $v=2$, and let $\tau$ be a primitive root modulo $p^2$. Suppose that the integer $v=2$ is a primitive root modulo $p$, but not modulo $p^2$. Then, the equation
\begin{equation}
\tau^{pn}-2=0
\end{equation}
has a solution in prime $p \geq 2$, and $n\geq1$ such that $\gcd(n,p-1)=1$ if and only if $2^{p-1}-1 \not \equiv 0 \bmod p$ and  $2^{p-1}-1 \equiv 0 \bmod p^2$. Constructing a indicator function, see Lemma \ref{lem33.5}, and sum it over the primes lead to
\begin{equation} \label{12.663}
 \sum_{ p \leq x} \Psi_v (p^2).
\end{equation}
Expanding the indicator function in (\ref{12.663}) yield\\
\begin{eqnarray} \label{12.605}
\sum_{p \leq x} \Psi_v (p^2) 
&=&\sum_{p \leq x,}\sum _{\gcd (n,p-1)=1} \frac{1}{\varphi(p^2)}\sum_{0\leq m< \varphi(p^2)} e^{\frac{i 2\pi  (\tau ^{pn}-v)m}{\varphi(p^2)} }\\
&=&\sum_{ p \leq x}\frac{1}{\varphi(p^2)} \sum_{\gcd (n,p-1)=1}1   \nonumber\\   
&& \qquad \qquad+    \sum_{ p \leq x,}\sum _{\gcd (n,p-1)=1} \frac{1}{\varphi(p^2)}\sum_{1\leq m<\varphi(p^2)} e^{\frac{i 2\pi  (\tau ^{pn}-v)m}{\varphi(p^2)} } \nonumber \\
&=&M_v(x)\quad + \quad E_v(x) \nonumber.
\end{eqnarray} 
The main term $M_v(x)$ is determined by the index $m=0$, and the error term $E_v(x)$ is determined by the range $1 \leq m< \varphi(p^2)$. Applying Lemma \ref{lem44.2} to the main term and applying Lemma \ref{lem55.1} to the error term yield
\begin{eqnarray} \label{12.608}
M_v(x) \quad +\quad  E_v(x) 
&=& c_v \log \log x     +O \left (\frac{1}{\log x}\right )   +O \left ((\log \log x)^{1-\varepsilon}\right )  \nonumber\\
&=& c_v \log \log x     +O \left ((\log \log x)^{1-\varepsilon}\right )  ,
\end{eqnarray} 
where $\varepsilon>0$ is a small number, and $c_v \geq 0$ is the density constant. 
\end{proof}

\begin{lem} \label{lem12.1} If $g$ is a primitive root modulo $p$, then $g+mp$ is a primitive root modulo $p^2$ for all $m \in [0,p-1]$ but one exceptional value.

\end{lem}
\begin{proof}
\cite[Theorem 2.5]{RH95}
\end{proof}

%49-49-49-49-49-49-49-49-49-49-49-49-49-49-49-49-49-49-49-49-49-49-49-49-49-49-49-49-49-49-49-49-49-49-49-49-49-49-49-49-49-49-49-49-49-

%ssssssssssssssssssssssssssssssssssssssssssssssssssssssssssssssssss
\section{The Order Series $\sum_{n \geq 2 } 1/n\ord_n(v)$}
The Romanoff problem is concerned with the evaluation of the series $\sum_{n \geq 2} 1/(n\ord(2))$. This seris occurs in the calculation of the density of the binary additive problem $n=p+2^k$. Much more general versions of this series are used in similar additive problems.

\subsection{Order Series Over The Integers} \label{sec6}
\begin{thm} \label{thm49.3}  {\normalfont ( \cite{MS96}) } Let $f_v(n) =\ord_{n}(v)$. Then
\begin{enumerate} [font=\normalfont, label=(\roman*)]
\item If $\varepsilon>0$ is an arbitrary small number, then there is an absolute constant $c_2$ for which,

\begin{equation}
\sum_{n \geq 2} \frac{1}{nf_v(n)^{\varepsilon}}\leq  e^{\gamma} \left( \log \log v +\varepsilon^{-1} +c_2\right ) .
\end{equation}

\item If $\varepsilon>0$ is an arbitrary small number, let $x \geq 2$, and let $v=1+\lcm[1,2, \ldots, x]$, then

\begin{equation}
\sum_{n \geq 2} \frac{1}{nf_v(n)^{\varepsilon}}\geq  e^{\gamma} \log \log v +O\left ( \log \log \log v\right ) .
\end{equation}
\end{enumerate}
\end{thm}

%bbbbbbbbbbbbbbbbbbbbbbbbbbbbbbbbbbbbbbbbbbbb
\subsection{Order Series Over Subsets Of Integers} \label{sec6}
Let $q\geq 2$ be a prime power, and let $v \geq$ be a fixed integer. The asymptotic formulas for the restrictions to relatively primes subsets of integers $\mathcal{A}=\{n\geq 1:\gcd(\ord_n(v),q)=1\}$ are considered in this section. These results are based on the counting function $A(x)=\{n \leq x:n \in \mathcal{A}\}$.
\begin{thm}  \label{thm13.1} {\normalfont (\cite[Theorem 4]{MH05})} For a prime power $q \geq 2$, and a large number \(x\geq 1,\) the counting function $A(x)$ has the asymptotic foirmula
\begin{equation}
\sum_{\substack{n\leq x\\ \gcd(\ord_n(v),q)=1}} 1=a(q,v)\frac{x}{\log^{c(q,v)} x} \left (1+O_q\left(\frac{(\log \log x)^5}{(\log x)^{c(q,v)+1}} \right ) \right ),
\end{equation}
where $a(q,v)>0$ and $c(q,v)>0$ are constants.
\end{thm}
\begin{thm}  \label{thm13.2} For any prime power $q\geq 2$, and fixed integer let $v \geq$, the order series converges:
\begin{equation}
\sum_{\substack{n\leq x\\ \gcd(\ord_n(v),q)=1}} \frac{1}{n \ord_n(v)}<  \infty.
\end{equation}
\end{thm}
\begin{proof} Let $\mathcal{A}=\{n\geq 1:\gcd(\ord_n(v),q)=1\}$ and let $A(x)=\{n \leq x:n \in \mathcal{A}\}$ be the corresponding the counting function. The series has an integral representation as
\begin{equation}
\sum_{\substack{n\leq x\\ \gcd(\ord_n(v),q)=1}} \frac{1}{n \ord_n(v)}=\int_1^{\infty} \frac{1}{t \ord_t(v)} dA(t).
\end{equation}
Use the bounds of the order function $1/ t< 1/\ord_t(v)<1/ \log t$ and its derivatives
\begin{equation}
-\frac{1}{t^2}< \frac{d}{dt} \frac{1}{\ord_t(v)}< -\frac{1}{t},
\end{equation}
and Theorem \ref{thm13.1}, which gives $)A(t)\ll x \log^{-c(q,v)} x $, to estimate the integral:
\begin{eqnarray}
\int_1^{\infty} \frac{1}{t \ord_t(v)} dA(t)
&=& \frac{A(t)}{t \ord_t(v)}- \int_1^{\infty} \left ( \frac{-1}{t^2 \ord_t(v)}+\frac{1}{t}\frac{d}{dt} \frac{1}{\ord_t(v)} \right ) A(t) dt\nonumber\\
&\ll&  O\left(\frac{1}{(\log x)^{c(q,v)+1}} \right )+    \int_1^{\infty} \left ( \frac{1}{t^2 \log t}+\frac{1}{t}\frac{1}{\log t}  \right )\frac{t}{\log^{c(q,v)} t} dt                 \\
&=&O\left(\frac{1}{(\log x)^{c(q,v)}} \right ) \nonumber, 
\end{eqnarray}
where \(c(q,v)>1\) is a  constant.
\end{proof}

%bbbbbbbbbbbbbbbbbbbbbbbbbbbbbbbbbbbbbbbbbbbbbbbbbb
\subsection{An Estimate For The Series $\sum_{p } \omega(p)$} 
The new result in Theorem \ref{thm1.2} is used to sharpen the numerical evaluation of the series
\begin{equation}
\sum_{p \geq 2} \frac{1}{\omega(p)} \leq 0.9091\ldots, 
\end{equation}
where $\omega(p)=\ord_{p^2}(2)$. The above estimate was computed in \cite{GS99}. Similar routines are used here too.

\begin{lem} \label{lem49.6}  Let $\omega(p) =\ord_{p^2}(2)$. Then
\begin{equation}
\sum_{p \geq 2} \frac{1}{\omega(p)}\leq  0.811049529055567378261719 \ldots .
\end{equation}
\end{lem}
\begin{proof} Start substituting the data 
\begin{enumerate} [font=\normalfont, label=(\roman*)]
\item $ \ord_{p^2}(2)= \ord_{p}(2)$   if $2^{p-1}-1 \equiv 0 \bmod p^2$,     and
\item $\ord_{p^2}(2)= p\ord_{p}(2) $   if $2^{p-1}-1 \not \equiv 0 \bmod p^2$,
\end{enumerate}
into the series:
\begin{eqnarray}
\sum_{p \geq 2 } \frac{1}{\omega(p)}&=&\sum_{\substack{p \geq 2 \\ 2^{p-1}-1\equiv 0 \bmod p^2}} \frac{1}{\omega(p)} +\sum_{\substack{p \geq 2 \\ 2^{p-1}-1\not \equiv 0 \bmod p^2}} \frac{1}{\omega(p)} \\
&=&\sum_{\substack{p \geq 2 \\ 2^{p-1}-1\equiv 0 \bmod p^2}} \frac{1}{\ord_p(2)} +\sum_{\substack{p \geq 2 \\ 2^{p-1}-1\not \equiv 0 \bmod p^2}} \frac{1}{p \ord_p(2)}  \nonumber.
\end{eqnarray}
Using $\ord_p(2) \geq \log p/ \log 2$, the upper bound  $W_2(x)\leq 8 \log \log x$, see Theorem \ref{thm1.2}, and the numerical data in \cite{GS99} and \cite{DK11}, set $x=7 \times 10^{15}$, the first subseries reduces to
\begin{eqnarray}
\sum_{\substack{p \geq 2 \\ 2^{p-1}-1\equiv 0 \bmod p^2}} \frac{1}{\ord_p(2)}&=&\sum_{\substack{p \leq 10^{15} \\ 2^{p-1}-1\equiv 0 \bmod p^2}} \frac{1}{\ord_p(2)} +\sum_{\substack{p >10^{15} \\ 2^{p-1}-1\equiv 0 \bmod p^2}} \frac{1}{\ord_p(2)}  \nonumber\\
&\leq& \frac{1}{\ord_{1093^2}(2)}+\frac{1}{\ord_{3511^2}(2)}+\sum_{\substack{p > 10^{15} \\ 2^{p-1}-1\equiv 0 \bmod p^2}} \frac{\log 2}{\log p} \nonumber\\
&\leq& \frac{1}{364}+\frac{1}{1755}+ \int_{10^{15}}^{\infty}\frac{\log 2}{\log t}d W_2(t)  \nonumber\\
&\leq& \frac{1}{364}+\frac{1}{1755}+  \frac{8 c\log 2 \log \log 10^{15}}{10^{15}}\nonumber\\
&\leq & 0.2766564971799087434188077, 
\end{eqnarray}
where $0<c\leq 10$ is a small constant. Fix a number $x =10^4$, then the second subseries reduces to
\begin{eqnarray}
\sum_{\substack{p \geq 2 \\ 2^{p-1}-1\not \equiv 0 \bmod p^2}} \frac{1}{p \ord_p(2)}&=&\sum_{\substack{p \leq x \\ 2^{p-1}-1\not \equiv 0 \bmod p^2}} \frac{1}{p \ord_p(2)} +\sum_{\substack{p >x \\ 2^{p-1}-1\not \equiv 0 \bmod p^2}} \frac{1}{p \ord_p(2)} \nonumber\\
&\leq&0.5343930318756586348429114 \ldots,
\end{eqnarray}
where the lower tail is computed by a computer algebra system:  
\begin{equation}
\sum_{p \leq x} \frac{1}{p\ord_p(2)}=0.3172457909240327210173469 \ldots,
\end{equation}
and the upper tail is estimated by an integral approximation:
\begin{equation}
 \sum_{p >x} \frac{\log 2}{p \log p}\leq \frac{2}{\log x} =0.2171472409516259138255645\ldots,
\end{equation}
\end{proof} 
For very large $x\geq 1$ the series is approximately
\begin{equation}
\sum_{p \geq 2} \frac{1}{\omega(p)}\leq .593902288103941464436155 + \frac{2}{\log x} .
\end{equation}
Thus, the numerical value can be reduced to $\sum_{p \geq 2} 1/\omega(p) \leq .624$ by increasing $x>10^{50}$.

%ppppppppppppppppppppppppppppppppppppppppppppppppppppp
\subsection{Problems}
\begin{enumerate}
\item Given an arbitrary small number $\varepsilon>0$ , use the upper bound $\ord_n(v)< n$ of the order modulo $n$ to show that
$$
\sum_{n \geq 2} \frac{1}{n \ord_n(v)^{\varepsilon}}\geq \zeta(1+)\prod_{p \mid v} \left(1- \frac{1}{p^{1+\varepsilon}}\right ) .
$$

\item Evaluate the squarefree oprdr series
$$
\sum_{n \leq 2} \frac{\mu(n)^2}{n \ord_n(v)}\geq \frac{6 e^{\gamma}}{\pi^2} \log \log v +O(1) .
$$

\item Evaluate the limit
$$
\lim_{m \to \infty} \sum_{\substack {n \geq 2 \\ \ord_n(v)=m}} \frac{1}{n}=0 .
$$
\item Evaluate the finite sum
$$
 \sum_{m \leq x} \sum_{\substack {n \geq 2 \\ \ord_n(v)=m}} \frac{1}{n}=a_v \log x +o(\log x),
$$
where $a_v$ is a constant.

\end{enumerate}

%13-13-13-13-13-13-13-13-13-13-13-13-13-13-13-13-13-13-13-13-13-13-13-13-13-13-13-13-13-13-13-13-13-
\newpage
\section{NonWieferich Primes}  \label{s13}
Let $\mathbb{P} =\{2,3,5,7, \ldots \}$ be the set of prime numbers, and let $v \geq 2$ be a fixed integer. The subsets of Wieferich primes and nonWieferich primes are defined by
\begin{equation}
\mathcal{W}_v=\{ p \in \mathbb{P}: v^{p-1}-1 \equiv 0 \bmod p^2\}
\end{equation}
and 
\begin{equation}
\overline{\mathcal{W}}_v=\{ p \in \mathbb{P}: v^{p-1}-1 \not \equiv 0 \bmod p^2\}\nonumber
\end{equation}
respectively. The set of primes $\mathbb{P}=\mathcal{W}_v \cup \overline{\mathcal{W}}_v$ is a disjoint union of these subsets.\\

The subsets $\mathcal{W}_v$ and $ \overline{\mathcal{W}}_v$ have other descriptions by means of the orders $ \ord_{p}(v) =d \mid p- 1$ and $\ord_{p^2}(v) =pd\ne p$ of the base $v$ in the cyclic group $\left (\Z/p \Z \right )^{\times}$ and $\left (\Z/p^2 \Z \right )^{\times}$ respectively. The order is defined by $\ord_n(v)= \text{min} \{ m\geq 1: v^{m-1}-1 \equiv 0 \bmod n\}$. Specifically,
\begin{equation} \label{68800}
	\mathcal{W}_v=\left\{ p: \ord_{p^2}(v) \mid p- 1 \right \}
\end{equation}
and
\begin{equation} \label{68804}
	\overline{\mathcal{W}}_v=\left\{ p:\ord_{p^2}(v)\nmid p- 1 \right \}\nonumber.
\end{equation}
For a large number $x \geq 1$, the corresponding counting functions for the number of such primes up to $x$ are defined by
\begin{equation} \label{68806}
	W_{v}(x)=\#\left\{ p\leq x: \ord_{p^2}(v) \mid p- 1\right \}
\end{equation}
and
\begin{equation} \label{68806}
	\overline{W}_{v}(x)=\pi(x)-W_v(x) \nonumber,
\end{equation}
where $\pi(x)=\#\{p \leq x \}$ is the primes counting function, respectively.\\

Assuming the $abc$ conjecture, several authors have proved that there are infinitely many nonWieferich primes, see \cite{SJ88}, \cite{GM13}, et alii. These results have lower bounds of the form
\begin{equation} \label{68822}
	\overline{W}_{v}(x) \gg \frac{\log x}{\log \log x}
\end{equation}
or slightly better. In addition, assuming the Erdos binary additive conjecture, there is a proof that the subset of nonWieferich primes has nonzero density. More precisely,
\begin{equation} \label{68822}
	\overline{W}_{v}(x) \geq c\frac{x}{\log x}
\end{equation}
where $c>0$ is a constant, see \cite[Theorem 1]{GS99} for the details. 

\subsection{Result For Nonzero Density}
Here, it is shown that the subset of nonWieferich primes has density 1 in the set of primes unconditionally.

\begin{thm} \label{thm13.1} Let $v\geq 2$ be a small base, and let $x \geq 1$ be a large number. Then, the number of nonWieferich primes has the asymptotic formula
\begin{equation} \label{13.822}
	\overline{W}_{v}(x) = \frac{x}{\log x}+O\left( \frac{x  }{\log^ 2x}\right ).
\end{equation}
\end{thm}
\begin{proof} The upper bound $W_v(x) \leq 4v \log \log x$, confer Theorem \ref{thm1.2}, is used below to derive a lower bound for the counting function 

$\overline{W}_{v}(x)$. This is as follows:
\begin{eqnarray}
\overline{W}_{v}(x)&=&\#\left\{ p\leq x:\ord_{p^2}(v) \nmid p- 1\right \} \nonumber \\
&=&\pi(x)-W_{v}(x) \\
&\geq& \pi(x)-4v\log \log x \nonumber \\
&=& \frac{x}{\log x}+O\left (\frac{x}{\log^2 x} \right ) \nonumber,
\end{eqnarray}
where $\pi(x)=\#\{p \leq x \}=x/\log x+O(x/\log^2 x)$.
\end{proof}

\begin{cor} \label{cor1.1} Almost every odd integer is a sum of a squarefree number and a power of two.
\end{cor}
\begin{proof} Same as Theorem 1 in \cite{GS99}, but use the upper bound  $W_v(x) \leq 4v \log \log x$.
\end{proof}

%11-11-11-11-11-11-11-11-11-11-11-11-11-11-11-11-11-11-11-11-11-11-11-11-11-11-11-11-11-11-
\newpage
\section{Generalizations} \label{s11}
The concept of Wieferich primes extends in many different directions, see \cite{VF00}, \cite{KN15}. \\

\begin{dfn} {\normalfont 
An integer $n \geq 1$ is a \textit{pseudoprime} to base $v \geq 2$ if the congruence $v^{n-1} \equiv 1 \bmod n$ holds. }
\end{dfn}
\begin{dfn} {\normalfont 
An integer $n \geq 1$ is a \textit{Carmichael number} if the congruence $v^{n-1} \equiv 1 \bmod n$ holds for every $v$ such that $\gcd(v,n)=1$. }
\end{dfn}

\begin{lem} \label{lem11.1} Every Carmichael number is squarefree and satisfies the followings properties.
\begin{enumerate}  [font=\normalfont, label=(\roman*)]
\item $n=q_1 q_2 \cdots q_t$ with $q_1,q_2 < \cdots <q_t$ primes in increasing order.
\item $q_i \leq \sqrt{n}$
\item $q-1 \mid n-1$ for every prime divisor $q \mid n$
\end{enumerate}
\end{lem}
 
\begin{lem} \label{lem11.2} {\normalfont (Cipolla)} 
Let \(n\geq 3\) be an integer, and let \(v \geq 2\) be a fixed integer. Then, the congruence $v^{n-1} \equiv 1 \bmod n$ has infinitely many solutions $n \geq 3$.	
\end{lem}
A proof appears in \cite[p. 125]{RD96}.\\

\begin{dfn} {\normalfont
An integer $n \geq 1$ is a \textit{Wieferich pseudoprime} to base $v \geq 2$ if the congruence $v^{n-1} \equiv 1 \bmod n^2 $ holds. }
\end{dfn}

Some information on the calculation of the constants $c_v$ for pseudoprimes is available in \cite{WS82}.\\

%16-16-16-16-16-16-16-1616-16-16-16-16-16-16-1616-16-16-16-16-16-16-1616-16-16-16-16-16-16-16
\newpage
\section{Counting Function For The Abel Numbers} \label{s16}
The subset of integers $\mathcal{A}_v=\left\{ n:\ord_{n^2}(v) \mid \lambda(n) \right \}$ associated with the base $v \geq 2$  congruence $v^{\lambda(n)}-1 \equiv 0 \bmod n^2$. These numbers will be referred to as Abel numbers to commemorate the earliest research on this topic, see \cite{KR05}, \cite[p.\ 413]{RP98}. The corresponding counting function is defined by
\begin{equation} \label{6002}
	A_{v}(x)=\#\left\{ n\leq x:\ord_{n^2}(v) \mid \varphi(n) \right \}.
\end{equation}
The heuristic argument in \cite[p.\ 413]{RP98} is not conclusive, but claims something as
\begin{equation} \label{6006}
	A_{v}(x) \approx  \sum_{p \leq x} \frac{1}{p} \ll \log x.
\end{equation}

%bbbbbbbbbbbbbbbbbbbbbbbbbbbbbbbbbbbbbbbbbbbbbbbbbbbbbbbbbbbbbbbbbbb
\subsection{Proof Of Theorem \ref{thm16.1}}
\begin{thm} \label{thm16.1} Let $v\geq 2$ be a base, and let  $x \geq 1$ and $z \geq x$ be large numbers. Then, the number of Abel numbers in the short interval $[x, x+z]$ has the asymptotic formula
\begin{equation} \label{1022}
	A_v(x+z)-	A_v(x)=c_v \left ( \log (x+z)-\log( x) \right )+E_v( x,z),
\end{equation}
where $c_v \geq 0$ is the correction factor, and $E_v(x,z)$ is an error term.
\end{thm} 

\begin{proof} Let \(x\geq 1\) be a large number, and fix an integer \(v\geq 2\). Consider the sum of the characteristic function for the fixed element $v$ of order $\ord_{n^2}(v) \mid \varphi(n)$ over the integers in the short interval \([x,x+z]\). Then \\
\begin{equation} \label{16.003}
A_{v}(x+z)-A_{v}(x)= \sum_{x \leq n \leq x+z} \Psi_0 (n^2).
\end{equation}\\
Replacing the characteristic function, see Lemma \ref{lem33.19}, and expanding the difference equation (\ref{16.003}) yield\\
\begin{eqnarray} \label{16.005}
\sum_{x \leq n \leq x+z} \Psi_0 (n^2) 
&=&\sum_{x \leq n \leq x+z,}\sum_{d \mid \lambda(n),} \sum _{\gcd (r,\lambda(n)/d)=1} \frac{1}{\varphi(n^2)}\sum_{0\leq m< \varphi(n^2)} e^{\frac{i 2\pi (\tau ^{nr}-v)m}{\varphi(n^2)} }\nonumber \\
&=&\sum_{x \leq n \leq x+z}\frac{1}{\varphi(n^2)} \sum_{d \mid \lambda(n),}\sum_{\gcd (r,\lambda(n)/d)=1}1   \nonumber\\   
&& +    \sum_{x \leq n \leq x+z,}\sum_{d \mid \lambda(n),}\sum _{\gcd (r,\lambda(n)/d)=1} \frac{1}{\varphi(n^2)}\sum_{1\leq m<\varphi(n^2)} e^{\frac{i 2\pi  (\tau ^{nr}-v)m}{\varphi(n^2)} } \\
&=&M_v(x,z)\quad + \quad E_v(x,z) \nonumber.
\end{eqnarray} 
The main term $M_v(x,z)$ is determined by the index $m=0$, and the error term $E_v(x,z)$ is determined by the range $1 \leq m< \varphi(n^2)$. Applying Lemma \ref{lem44.1} to the main term and applying Lemma \ref{lem55.3} to the error term yield\\
\begin{eqnarray} \label{16.008}
& &  M_v(x,z) \quad +\quad  E_v(x,z) \\
&=& a_v \left ( \log (x+z)- \log (x) \right )    +O \left (\frac{1}{\log x}\right )   +O \left (\frac{z^{1/2}}{x^{1/2}\log x}\right )  \nonumber,
\end{eqnarray} 
Next, assuming that $z =O(x)$ it reduces to
\begin{equation} \label{16.008}
A_{v}(x+z)-A_{v}(x)= a_v\left (  \log (x+z)-g \log (x) \right ) +O \left (\frac{1}{\log x}\right ),
\end{equation} 
where $a_v \geq 0$ is the density constant.
\end{proof}

The specific constant $a_v\geq 0$ for a given fixed base $v\geq 2$ is a problem in algebraic number theory, see Theorem \ref{thm77.1} for some details. 

\newpage
%\section*{References}


\begin{thebibliography}{999}
\bibitem{AS97}  Agoh, Takashi; Dilcher, Karl; Skula, Ladislav Wilson quotients for composite moduli. Math. Comp. 67 (1998), no. 222, 843-861. 
\bibitem{AC14} Ambrose, Christopher Daniel. On Artin's primitive root conjecture.  Doctoral thesis, 2014.
\bibitem{AP76} Apostol, Tom M. Introduction to analytic number theory. Undergraduate Texts in Mathematics. Springer-Verlag, New York-Heidelberg, 1976.
\bibitem{BD11} Balog, Antal; Cojocaru, Alina-Carmen; David, Chantal. Average twin prime conjecture for elliptic curves. Amer. J. Math. 133 (2011), no. 5, 1179-1229.
\bibitem{BJ17} Julio Brau, Character sums for elliptic curve densities, arXiv:1703.04154.
\bibitem{BS05} W.D. Banks, F. Luca, F. Saidak and P. Stanica. Compositions with the Euler and Carmichael Functions, Abh. Math. Sem. Univ. Hamburg., 75 (2005), 215-244.
\bibitem{BV01} Bilu, Yu.; Hanrot, G.; Voutier, P. M. Existence of primitive divisors of Lucas and Lehmer numbers. With an appendix by M. Mignotte. J. Reine Angew. Math. 539 (2001), 75-122. 
\bibitem{CP97} Crandall, Richard; Dilcher, Karl; Pomerance, Carl. A search for Wieferich and Wilson primes. Math. Comp. 66 (1997), no. 217, 433-449. 
\bibitem{CP05}  Crandall, Richard; Pomerance, Carl. Prime numbers. A computational perspective. Second edition. Springer, New York, 2005.
\bibitem{CP09} Peter J. Cameron and D. A. Preece, Notes on primitive lambda-roots, \text{http:www.maths.qmul.ac.ukpjccsgnoteslambda.pdf}.
\bibitem{CR07} Carmichael, R. D.; On Euler's $\varphi$-function. Bull. Amer. Math. Soc. 13 (1907), no. 5, 241-243.
\bibitem{CR10} Carmichael, R. D. Note on a new number theory function. Bull. Amer. Math. Soc. 16 (1910), no. 5, 232-238. 
\bibitem{DK11} Dorais, Francois G.; Klyve, Dominic. A Wieferich prime search up to $6.7 \times 10^{15}$. J. Integer Seq.  14  (2011),  no. 9, Article 11.9.2, 14 pp. 
\bibitem{DD07} DeKoninck, J.-M.; Doyon, N. On the set of Wieferich primes and of its complement. Ann. Univ. Sci. Budapest. Sect. Comput. 27 (2007), 3-13. 
\bibitem{DP99} Dusart, Pierre Inegalites explicites pour $psi(x)$, $\theta(x)$, and $\pi(x)$ et les nombres premiers. C. R. Math. Acad. Sci. Soc. R. Can. 21 (1999), no. 2, 53-59.
\bibitem{ES91} Erdos, Paul; Pomerance, Carl; Schmutz, Eric. Carmichael's lambda function. Acta Arith. 58 (1991), no. 4, 363-385.
\bibitem{EW03} Everest, Graham; van der Poorten, Alf; Shparlinski, Igor; Ward, Thomas. Recurrence sequences. Mathematical Surveys and Monographs, 104. American Mathematical Society, Providence, RI, 2003. 
\bibitem{GC86}Gauss, Carl Friedrich. Disquisitiones arithmeticae. Translated by Arthur A. Clarke. Revised by William C. Waterhouse, Cornelius Greither and A. W. Grootendorst. Springer-Verlag, New York, 1986. 
\bibitem{GM13} H. Graves, M. Ram Murty, The abc conjecture and non-Wieferich primes in arithmetic progressions, J. Number Theory 133 (2013) 1809-1813.
\bibitem{GS99} Granville, Andrew; Soundararajan, K. A binary additive problem of Erdos and the order of $2 \bmod p^2$. Ramanujan J. 2 (1998), no. 1-2, 283-298.
\bibitem{HG83} Hinz, Jurgen G. Character sums and primitive roots in algebraic number fields. Monatsh. Math. 95 (1983), no. 4, 275-286.
\bibitem{HC67} C. Hooley, On Artins conjecture, J. Reine Angew. Math. 225, 209-220, 1967.
\bibitem{JG73} Janusz, Gerald J. Algebraic number fields. Pure and Applied Mathematics, Vol. 55. Academic Press, New York-London, 1973. 
\bibitem{JJ71} Johnsen, John. On the distribution of powers in finite fields. J. Reine Angew. Math. 251 1971 10-19. 
\bibitem{KD07} Klyve, Dominic Explicit bounds on twin primes and Brun's Constant. Thesis (Ph.D.)-Dartmouth College. 2007.
\bibitem{KN15} Katz, Nicholas M. Wieferich past and future. Topics in finite fields, 253-270, Contemp. Math., 632, Amer. Math. Soc., Providence, RI, 2015. 
\bibitem{KM16} Srinivas Kotyada, Subramani Muthukrishnan, Non-Wieferich primes in number fields and ABC conjecture, arXiv:1610.00488.
Computers in Mathematical Research, North-Holland, 1968, pp. 84-88.
\bibitem{KR05} Wilfrid Keller, Jorg Richstein, Solutions of the congruence $a^{p-1} \equiv 1 \bmod p^r$, Math. Comp. 74 (2005), 927-936. 
\bibitem{KJ05} Joshua Knauer, and Jorg Richstein, The continuing search for Wieferich primes, Math. Comp. 74 (2005), 1559-1563.     
\bibitem{KS13} James, Kevin; Smith, Ethan. Average Frobenius distribution for the degree two primes of a number field. Math. Proc. Cambridge Philos. Soc. 154 (2013), no. 3, 499-525.
\bibitem{LE38} Lehmer, Emma On congruences involving Bernoulli numbers and the quotients of Fermat and Wilson. Ann. of Math. (2) 39 (1938), no. 2, 350-360.
\bibitem{LJ94} Lagarias, J. C. Errata to: The set of primes dividing the Lucas numbers has density 2/3, Pacific J. Math. 118 (1985), no. 2, 449-461; MR0789184. Pacific J. Math. 162 (1994), no. 2, 393-396. 
\bibitem{LN97} Lidl, Rudolf; Niederreiter, Harald. Finite fields. With a foreword by P. M. Cohn. Second edition. Encyclopedia of Mathematics and its Applications, 20. Cambridge University Press, Cambridge, 1997.
\bibitem{LS14} Lenstra, H. W., Jr.; Stevenhagen, P.; Moree, P. Character sums for primitive root densities. Math. Proc. Cambridge Philos. Soc. 157 (2014), no. 3, 489-511. 
\bibitem{LZ10}  Languasco, Alessandro; Zaccagnini, Alessandro. Computing the Mertens and Meissel-Mertens constants for sums over 
arithmetic progressions. Experiment. Math. 19 (2010), no. 3, 279-284.
\bibitem{ML72} Mordell, L. J. On the exponential sum $\sum_{1\leq x \leq X} exp (2\pi i(ax+bg^x )/p)$. {\it Mathematika}  {\bf 19}  (1972), 84-87.
\bibitem{MC05} Moreno, Carlos Julio. Advanced analytic number theory: L-functions. Mathematical Surveys and Monographs, 115. 
American Mathematical Society, Providence, RI, 2005.
\bibitem{MH05}  Muller, Helmut On the distribution of the orders of 2(modu) for odd u. Arch. Math. (Basel) 84 (2005), no. 5, 412-420.
\bibitem{MP03} Mihailescu, Preda. A class number free criterion for Catalan's conjecture. J. Number Theory  99  (2003),  no. 2, 225-231.
\bibitem{MP04} Greg Martin, Carl Pomerance. The iterated Carmichael $\lambda$-function and the number of cycles of 
the power generator, arXiv:math/0406335.
\bibitem{MP93}  Montgomery, Peter L. New solutions of $a^{p-1} \equiv 1 \bmod p^2$. Math. Comp. 61 (1993), no. 203, 361-363.
\bibitem{MS04}  Muller, Thomas W.; Schlage-Puchta, Jan-Christoph On the number of primitive $\lambda$-roots. Acta Arith.  115  (2004),  no. 3, 217-223.
\bibitem{MS96} Murty, M. Ram; Rosen, Michael; Silverman, Joseph H. Variations on a theme of Romanoff. Internat. J. Math. 7 (1996), no. 3, 373-391. 
\bibitem{PA09} Paszkiewicz, A. A new prime $p$ for which the least primitive root $\bmod p$ and the least primitive root $\bmod p^2$   are not equal. Math. Comp.  78  (2009),  no. 266, 1193-1195. 
\bibitem{PW14} Palenstijn, Willem Jan. Radicals in Arithmetic, Leiden University dissertation, 2014.
\bibitem{RH95} Rose, H. E. A course in number theory. Second edition. Oxford Science Publications. The Clarendon Press, Oxford University Press, New York, 1994.
\bibitem{RD96} Redmond, Don. Number theory. An introduction. Monographs and Textbooks in Pure and Applied Mathematics, 201. Marcel Dekker, Inc., New York, 1996.
\bibitem{RP98} Ribenboim, Paulo, The new book of prime number records, Berlin, New York: Springer-Verlag, 1996.
\bibitem{RS62}  Rosser, J. Barkley; Schoenfeld, Lowell Approximate formulas for some functions of prime numbers. Illinois J. Math. 6 1962 64-94.
\bibitem{SG02} Sebah, Pascal; Gourdon, Xavier. "Introduction to twin primes and Brun's constant computation", Preprint 2002.
\bibitem{SJ88} J. H. Silverman, Wieferich  criterion  and the  abc-conjecture, J. Number Theory 30 (1988) no. 2, 226-237.
\bibitem{SP03} Stevenhagen, Peter. The correction factor in Artin's primitive root conjecture. Les XXII emes Journees Arithmetiques (Lille, 2001). J. Theor. Nombres Bordeaux 15 (2003), no. 1, 383-391.   
\bibitem{SP69} Stephens, P. J. An average result for Artin conjecture. Mathematika 16, (1969), 178-188.
\bibitem{VR73} Vaughan, R. C. Some applications of Montgomery's sieve. J. Number Theory 5 (1973), 64-79.
\bibitem{VF00} Voloch, Jose Felipe. Elliptic Wieferich primes. J. Number Theory 81 (2000), no. 2, 205-209. 
\bibitem{WA09} A. Wieferich, Zum letzten Fermatschen Theorem , J. Reine Angew. Math. 136 (1909), 293-302.
\bibitem{WS82} Wagstaff, Samuel S., Jr. Pseudoprimes and a generalization of Artin's conjecture. Acta Arith. 41 (1982), no. 2, 141-150.  
\end{thebibliography}
\end{document}